\documentclass{amsart}

\usepackage{amsmath}
\usepackage{amssymb}
\usepackage{amsthm}
\usepackage[english]{babel}
\usepackage{bm}
\usepackage[margin=0.5in]{caption}
\usepackage{dsfont}
\usepackage{enumitem}
\usepackage[T1]{fontenc}
\usepackage{graphicx}
\usepackage[hidelinks]{hyperref} \usepackage[capitalise]{cleveref}
\usepackage{indentfirst}
\usepackage[utf8]{inputenc}
\usepackage{layout}
\usepackage{listings}
\usepackage{lscape}
\usepackage{mathrsfs}
\usepackage{mathtools}
\usepackage{nameref}
\usepackage[new]{old-arrows}
\usepackage{setspace}
\usepackage{stmaryrd}
\usepackage{textcomp}
\usepackage{thmtools}
\usepackage{xcolor}
\usepackage{xparse}

\usepackage{tikz}
\usepackage{tikz-cd}

\title{Relative Untwisted Outer Space for Right-Angled Artin Groups}
\author{Adrien \textsc{Abgrall}}
\date{\today}

\begin{document}

\pagestyle{plain}
\frenchspacing
\parindent=15pt
\theoremstyle{plain}
\newtheorem{thm}{Theorem}[section]
\newtheorem{lm}[thm]{Lemma}
\newtheorem{pro}[thm]{Proposition}
\newtheorem{cor}[thm]{Corollary}
\newtheorem{conj}[thm]{Conjecture}
\newtheorem*{csvdefi}{Definition}
\newtheorem*{csvlm}{Lemma}
\newtheorem*{csvcor}{Corollary}
\newtheorem*{csvthm}{Theorem}
\newtheorem{innercustomthm}{Theorem}
\newenvironment{customthm}[1]
  {\renewcommand\theinnercustomthm{#1}\innercustomthm}
  {\endinnercustomthm}
\newtheorem{innercustomcor}{Corollary}
\newenvironment{customcor}[1]
  {\renewcommand\theinnercustomcor{#1}\innercustomcor}
  {\endinnercustomcor}
\theoremstyle{definition}
\newtheorem{defi}[thm]{Definition}
\newtheorem{rem}[thm]{Remark}
\theoremstyle{remark}
\newtheorem{q}[thm]{Question}
\newtheorem{ex}[thm]{Example}
\newtheorem{nota}[thm]{Notation}
\newtheorem*{ack}{Acknowledgments}

\newcommand{\NN}{\mathbb{N}}
\newcommand{\ZZ}{\mathbb{Z}}
\newcommand{\QQ}{\mathbb{Q}}
\newcommand{\RR}{\mathbb{R}}
\newcommand{\CC}{\mathbb{C}}
\renewcommand{\SS}{\mathbb{S}}
\newcommand{\TT}{\mathbb{T}}
\newcommand{\HH}{\mathbb{H}}
\newcommand{\FF}{\mathbb{F}}

\newcommand{\cat}{\mathrm{CAT}(0)}
\newcommand{\s}{\mathrm{Sep}}
\newcommand{\out}{\mathsf{Out}}
\newcommand{\aut}{\mathsf{Aut}}
\newcommand{\inn}{\mathsf{Inn}}
\newcommand{\unt}{\mathcal{U}}
\newcommand{\uaut}{\mathcal{U}\mathsf{Aut}}
\renewcommand{\H}{\mathfrak{H}}
\newcommand{\K}{\mathfrak{K}}
\newcommand{\W}{\mathfrak{W}}
\renewcommand{\emptyset}{\varnothing}

\newcommand{\n}[1]{\left|\left|\,#1\,\right|\right|}
\newcommand{\floor}[1]{\left\lfloor #1 \right\rfloor}
\newcommand{\ceil}[1]{\left\lceil #1 \right\rceil}
\newcommand{\gen}[1]{\left\langle #1 \right\rangle}
\newcommand{\card}[1]{\left | #1 \right |}
\newcommand{\gl}{GL}
\newcommand{\uc}[1]{\widetilde{#1}}
\newcommand{\USS}{\uc{\SS}}
\newcommand{\sd}{\mathbin{\triangle}}

\newcommand{\bin}[1]{}
\newcommand{\todo}[1]{{\color{red}#1}}

\begin{abstract}

For $\mathcal{G} = \{G_i\}$, $\mathcal{H}=\{H_j\}$ two finite collections of finitely generated subgroups of a right-angled Artin group $A_\Gamma$, the untwisted McCool group $\unt(A_\Gamma;\,\mathcal{G},\mathcal{H}^t)$ is the subgroup of untwisted outer automorphisms of $A_\Gamma$ preserving the conjugacy class of each $G_i$ and acting trivially up to conjugacy on each $H_j$. We prove that when the $G_i$ are standard subgroups of $A_\Gamma$, $\unt(A_\Gamma;\,\mathcal{G},\mathcal{H}^t)$ acts properly cocompactly on a finite-dimensional subcomplex of the spine of untwisted outer space for $A_\Gamma$, providing a geometric model for this group and proof that it is of type VF.
\end{abstract}

\maketitle

\section{Introduction}

A right-angled Artin group is a group $A_\Gamma$ with a fixed finite \emph{standard generating set} $V$, where the only relations come from commutations of some pairs of standard generators, encoded by a graph $\Gamma$ with vertex set $V$. Examples include finite-rank free groups $F_n$ and free abelian groups $\ZZ^n$. Many properties of right-angled Artin groups are reminiscent of one class or the other. A finite generating set for the automorphism group $\aut(A_\Gamma)$ was found by Laurence (\cite{laurence}) and Servatius (\cite{servatius}). This set closely resembles the Nielsen generators for $\aut(F_n)$, but includes some automorphisms called \emph{twists} whose behaviour is similar to transvections in
$SL_n(\ZZ) = \aut(\ZZ^n)$. This observation warranted the definition of the \emph{untwisted subgroup} $\uaut(A_\Gamma)$ of $\aut(A_\Gamma)$, generated by all the Laurence-Servatius generators except the twists. It appears sometimes in the literature as the subgroup generated by \emph{long-range automorphisms}. Note that, while the definition of the untwisted subgroup might seem somewhat arbitrary, Fioravanti (\cite{coarsemedian}) characterized it as the group of automorphisms of $A_\Gamma$ that coarsely preserve a natural median structure, providing some geometric significance to this subgroup. Moreover, if for all $v\neq w\in V$, there is no inclusion between balls of radius $1$ around $v$ and $w$ in $\Gamma$, then no twist can exist. In that case, the untwisted subgroup is $\aut(A_\Gamma)$ itself. Because of the similarities in their generating sets, the untwisted subgroup $\uaut(A_\Gamma)$, and its outer automorphism image $\unt(A_\Gamma)\leq \out(A_\Gamma)$ are expected to share many properties with $\aut(F_n)$ and $\out(F_n)$ respectively.

The \emph{spine of outer space} for the free group $F_n$ is a simplicial complex constructed by Culler and Vogtmann (\cite{cullervogtmann}). It is contractible and admits a properly discontinuous and cocompact action of $\out(F_n)$ by simplicial isometries. Closely related is the \emph{spine of auter space} for $F_n$, constructed by Hatcher and Vogtmann (\cite{auterspace}), having an action of $\aut(F_n)$ with analogous properties. Recently, Charney, Stambaugh, and Vogtmann constructed in \cite{untwistedos} a \emph{spine of untwisted outer space} $K_\Gamma$ for $A_\Gamma$, having an action of the group of untwisted outer automorphisms $\unt(A_\Gamma)$ with analogous properties again. This construction provided a geometrically explicit proof of the fact $\unt(A_\Gamma)$ is of \emph{type VF}, namely that a torsion-free subgroup of finite index acts freely and cocompactly on $K_\Gamma$, a finite-dimensional contractible complex, with many homological consequences. In particular, it recovered the fact that $\unt(A_\Gamma)$ is finitely presented, which was known from an earlier work of Day (\cite{daypresentation}, Proposition~5.4). In \cite{spatial}, we gave a new construction of the spine $K_\Gamma$: vertices of $K_\Gamma$ are certain locally $\cat$ cube complexes called \emph{spatial cube complexes}, endowed with \emph{untwisted markings} identifying their fundamental group with $A_\Gamma$, considered modulo marking-preserving combinatorial isomorphism. Edges of $K_\Gamma$ correspond to marking-preserving \emph{hyperplane collapses} between spatial cube complexes. So far, no analogue to the spine of auter space was known for $A_\Gamma$. Even more recently, Bregman, Charney, and Vogtmann in \cite{twistedos} modified the construction of $K_\Gamma$ into a (significantly more complex) \emph{outer space for $A_\Gamma$}, i.e.~a contractible space with a properly discontinuous action of the whole $\out(A_\Gamma)$.

\bigskip

The group $\out(A_\Gamma)$ and its subgroup $\unt(A_\Gamma)$ act on the set of conjugacy classes of elements of $A_\Gamma$. The intent of this article is to obtain a geometric understanding of the stabilizer in $\unt(A_\Gamma)$ of a finite array of such conjugacy classes. Examples of such stabilizers are many, but include notably, when $A_\Gamma$ is free, pure (extended) mapping class groups of finite-type non-compact surfaces, and for a general $A_\Gamma$, the subgroup of \emph{outer conjugating automorphisms}, which are the outer automorphisms preserving the conjugacy class of every standard generator (see Example~\ref{conjugatingautomorphisms}). We obtain the following theorem.

\begin{customthm}{\ref{mccoolcyclic}}[special case]
Let $[h_1],\dots, [h_n]$ be finitely many conjugacy classes of elements of $A_\Gamma$. There exists a contractible subcomplex of the spine of untwisted outer space $K_\Gamma$ on which the untwisted stabilizer $Stab([h_1],\dots, [h_n])\cap \unt(A_\Gamma)$ acts properly and cocompactly. In particular, this group is of type VF.
\end{customthm}

To the best of our knowledge, the question remains unanswered when twists are allowed. Stabilizers in $\unt(A_\Gamma)$ have proved more convenient to grasp than their counterparts in $\out(A_\Gamma)$ due to the combinatorial nature of the spine of untwisted outer space. The question had been previously solved when $A_\Gamma$ is free: a theorem of McCool (\cite{mccooloriginal}) proves finite presentation of the stabilizer, a result of Guirardel and Levitt (\cite{guirardel_mccool_2015}) proves that it has type VF, and a recent work of Bestvina, Feighn and Handel (\cite{mccoolwhitehead}) produces a contractible subcomplex of Culler-Vogtmann's outer space with a cocompact action (see below). Their algorithmic proof relates to an algorithm announced by Gersten (\cite{gerstenalgorithmfree}) and verified by Kalajdžievski (\cite{algorithmfree}). Very recently, this special case has been obtained independently by Corrigan (\cite{corrigan}).

The results of Guirardel and Levitt, and of Bestvina, Feighn and Handel are even more general and invite to consider stabilizers not only of conjugacy classes of elements but also of conjugacy classes of subgroups. The relevant notion of a \emph{McCool group}, a subgroup of an outer automorphism group was introduced by Guirardel and Levitt.
\begin{defi}[\cite{mccooldef}]
    Let $G$ be a group and $\mathcal{G} = \{G_i\},\mathcal{H} = \{H_j\}$ two collections of subgroups of $G$. The \emph{McCool group} $\out(G;\,\mathcal{G},\mathcal{H}^t)\leq \out(G)$ is the subgroup of outer automorphisms that preserve the conjugacy class of each $G_i$ and are inner on each $H_j$ (i.e.~for every $j$, some representative coincides with the identity on $H_j$).

    When $G=A_\Gamma$ is a right-angled Artin group, we denote $\unt(A_\Gamma;\,\mathcal{G},\mathcal{H}^t)$ the intersection $\unt(A_\Gamma)\cap \out(A_\Gamma;\,\mathcal{G},\mathcal{H}^t)$.
\end{defi}

With this notation, the full generality of the result of Bestvina, Feighn and Handel is the following.

\begin{thm}[\cite{mccoolwhitehead}]
    For $\mathcal{G}$ a finite collection of finitely generated subgroups, $\out(F_n;\mathcal{G},\emptyset)$ acts properly and cocompactly on a subcomplex of the spine of Culler-Vogtmann's outer space for $F_n$. In particular, this group is of type VF.
\end{thm}

Our main theorem adapts this result to untwisted McCool subgroups of for right-angled Artin groups, with finitely many fixed (conjugacy classes of) finitely generated subgroups, and stabilized \emph{standard subgroups} (i.e.~subgroups generated by a subset of the standard generators).

\begin{customthm}{\ref{mccoolgeneral}}
Let $\mathcal{G}$ be a collection of standard subgroups of $A_\Gamma$ and let $\mathcal{H}$ be a finite collection of arbitrary finitely generated subgroups of $A_\Gamma$. There exists a contractible subcomplex of the spine of untwisted outer space $K_\Gamma$ on which $\unt(A_\Gamma; \mathcal{G},\mathcal{H}^t)$ acts properly and cocompactly. In particular, this group is of type VF.
\end{customthm}

The special case above corresponds to the case where $\mathcal{G} =\emptyset$ and $\mathcal{H}$ is made of cyclic subgroups. Passing from these cyclic subgroups to finitely generated ones (Section~\ref{fixedfg}) relies on a purely geometric translation length argument, developped in Section~\ref{minsets}. However, stabilization of standard subgroups requires much care, and is in fact the most technical part of this article (Section~\ref{relative}). 

Previous work of Day and Wade (\cite{relativeout}) had established type VF for $\out(A_\Gamma;\,\mathcal{G},\mathcal{H}^t)$ (not the untwisted version) but only under the extra assumption that every subgroup of $\mathcal{H}$ is standard as well. Moreover their proof is algebraic in nature and does not provide explicitly a geometric model for the group.

We make the observation that the untwisted automorphism group $\uaut(A_\Gamma)$ is isomorphic to the untwisted McCool group $\unt(A_\Gamma*\ZZ, \{A_\Gamma\}, \{\ZZ\}^t)$, and obtain as a direct corollary a \emph{spine of untwisted auter space for $A_\Gamma$}:

\begin{customcor}{\ref{auterspace}}
There exists a contractible simplicial complex $L_\Gamma$ on which $\uaut(A_\Gamma)$ acts properly and cocompactly. In particular, this group is of type VF.
\end{customcor}

\bigskip

In Section~\ref{background}, we recall all the needed concepts from the theory of $\cat$ cube complexes, right-angled Artin groups, hyperplane collapses and untwisted outer space. We prove a relative version of Droms' theorem on isomorphisms of right-angled Artin groups (Lemma~\ref{relativedroms}) and a generation result for certain subgroups of the untwisted outer automorphism groups (Proposition~\ref{fioravanti}).

In Section~\ref{minsets}, we provide results on the behavior of minsets of isometries acting cospecially on $\cat$ cube complexes, and minimal invariant subcomplexes for groups acting on such complexes.

In Section~\ref{collapses}, we give two preliminary results about how hyperplane collapses behave on locally convex subcomplexes.

In Section~\ref{relative}, we describe the spine of untwisted outer space relative to a collection $\mathcal{G}$ of standard subgroups as a subspace $K^\mathcal{G}$ of $K_\Gamma$ and prove several properties of $K^\mathcal{G}$.

In Section~\ref{fixedZ}, we introduce a further subspace $K^\mathcal{G}_\mathcal{H}$ of $K_\Gamma$ associated to a family $\mathcal{G}$ of stabilized standard subgroups and a family $\mathcal{H}$ of fixed arbitrary cyclic subgroups. The definition is close to the work of Bestvina, Feighn and Handel, but some new arguments are required.

In Section~\ref{contractibility}, we prove that $K^\mathcal{G}_\mathcal{H}$ is contractible, finishing the proof of Theorem~\ref{mccoolcyclic}. This proof of contractibility mirrors the proof of \cite{untwistedos}, however, more arguments are needed at each step to account for the new fixed subgroups.

Finally, in Section~\ref{fixedfg}, we use a short geometric argument to obtain the more general Theorem~\ref{mccoolgeneral}, and conjecture further generalizations.

\begin{ack}
I gratefully acknowledge support from project ANR-22-CE40-0004. I thank Samuel Fisher and Zachary Munro for many interesting discussions while writing this work. I am extremely grateful to my advisor, Vincent Guirardel, for his careful attention and precious advice.
\end{ack}

\section{Background}
\label{background}

\subsection{CAT(0) cube complexes}

The study of $\cat$ cube complexes started with Gromov (\cite{gromov}) and Sageev (\cite{ccc}). Further work is due to Haglund and Paulin (\cite{wallspace}) from the viewpoint of \emph{wallspaces}, and to Chepoi (\cite{median}) from the viewpoint of \emph{median graphs}. We will mainly use the same background results given in \cite{spatial}, Section~2.1 and refer the reader to them first. We use the same notation $\mu \colon X^{(0)}\times X^{(0)}\times X^{(0)}\to X^{(0)}$ for the \emph{median} defined on the $0$-skeleton of a $\cat$ cube complex $X$. Recall that when $X$ is $\cat$, a subcomplex $A$ of $X$ is \emph{convex} when $A$ is full and geodesic edge paths in $X$ joining vertices of $A$ are fully contained in $A$. When $X$ is locally $\cat$, a subcomplex of $X$ is \emph{locally convex} if any of its preimages in the universal cover of $X$ is convex. We also define \emph{hyperplanes} in locally $\cat$ cube complexes as equivalence classes of edges under parallelism. Hyperplanes are \emph{transverse} when they have dual edges spanning a square: a hyperplane might be transverse to itself. When $X$ is $\cat$ and $A, B$ are two subsets of $X$, we write $\s(A\mid B)$ to denote the set of hyperplanes of $X$ separating $A$ from $B$. We quote the following classical result.

\begin{lm}
\label{convexcarac}
Let $X$ be a $\cat$ cube complex, and $A\subseteq X$ a full subcomplex. The following are equivalent:
\begin{enumerate}
    \item $A$ is convex
    \item $\mu(A^{(0)},A^{(0)},X^{(0)})\subseteq A^{(0)}$
    \item $A$ is \emph{gated}: for every $x\in X^{(0)}$ there exists $\pi_A(x)\in A^{(0)}$ such that for every $a\in A^{(0)}$, $d(x,a) = d(x,\pi_A(x)) + d(\pi_A(x),a)$, i.e.~$\pi_A(x)$ lies on some gedoesic edge path joining $x$ and $a$.
    \item $A$ is an intersection of halfspaces
\end{enumerate}
Moreover, when they hold, $\pi_A(x)$ is unique and called the \emph{projection of $x$ onto $A$}.
\end{lm}
\begin{proof}
$(1.)\Rightarrow (2.)$: If $a,b\in A^{(0)}, x\in X^{(0)}$ then $\mu(a,b,x)$ belongs to a geodesic joining $a$ and $b$.

$(2.)\Rightarrow (3.)$: Let $x\in X$ be a vertex and set $\pi_A(x)$ to be any vertex in $A$ at minimal distance from $x$. Let $a\in A$ be an arbitrary vertex. The median $\mu(x,\pi_A(x),a)$ belongs to $A$ by assumption, and to a geodesic edge path joining $x$ and $\pi_A(x)$ by definition. Therefore it is equal to $\pi_A(x)$ by minimality. Thus $\pi_A(x)$ lies on some geodesic edge path joining $x$ and $a$.

$(3.)\Rightarrow (4.)$: Let $B\supseteq A$ be the intersection of all the halfpaces of $X$ that contain $A$ entirely. Assume $B\setminus A$ contains a vertex $x$. Since $x\notin A$, $d(\pi_A(x),x)>0$, hence there exists $H\in \s(x\mid \pi_A(x))$. Because $H$ separates two points of $B$, it must separate two points of $A$, hence the existence of $a\in A$ such that $H\in \s(x,a\mid \pi_A(x))$. Thus, a geodesic joining $x$ and $a$ through $\pi_A(x)$ crosses $H$ twice, a contradiction. Therefore, every vertex of $B$ is contained in $A$, which is full, thus $A$ is equal to $B$, which is an intersection of halfspaces.

$(4.)\Rightarrow (1.)$ is clear because halfspaces are convex.

For the uniqueness of the projection, any vertex $p$ with the same property lies on a geodesic joining $x$ and $\pi_A(x)$, but cannot be closer to $x$ by minimality of $d(x,\pi_A(x))$. Hence $p=\pi_A(x)$.
\end{proof}

\subsection{Right-angled Artin groups and their autmorphisms}
Let $\Gamma = (V,E)$ be a non-empty finite simple graph fixed for the rest of this article. The following definition gathers most of the tools we will need concerning right-angled Artin groups. A more thorough introduction appears in \cite{introraags}.

\begin{defi}
The \emph{right-angled Artin group} $A_\Gamma$ associated with $\Gamma$ is the group given by the following presentation:
\[A_\Gamma=\gen{V\mid [v_i,v_j]\; \forall \{v_i,v_j\}\in E}\]

Elements of $V\subset A_\Gamma$ are the \emph{standard generators}. The \emph{Salvetti complex} $\SS$ associated with $A_\Gamma$ is the subcomplex of the torus $(S^1)^V$ whose cubes correspond one-to-one to the cliques of $\Gamma$. Its fundamental group canonically identifies with $A_\Gamma$. Moreover, $\SS$ is a special (see \cite{special}), locally $\cat$ cube complex with a single vertex. Conversely every special cube complex with a single vertex and fundamental group isomorphic to $A_\Gamma$ is combinatorially isomorphic to $\SS$. The hyperplanes of $\SS$ are in one-to-one correspondence with $V$, and two hyperplanes are transverse if and only if their corresponding vertices are adjacent. Let $*$ denote the unique vertex of $\SS$ and $\USS$ the universal cover of $\SS$ with the pulled back labeling of hyperplanes by $V$. It is a $\cat$ cube complex with a natural free cocompact cospecial action of $A_\Gamma$. In particular, $\SS$ is a classifying space for $A_\Gamma$. Since $\SS$ is compact, $A_\Gamma$ is torsion-free.

Let $V^{\pm}=V\sqcup V^{-1}$ denote the set of standard generators and their inverses. For $x\in V^\pm$, the \emph{link of $x$} is the set $lk(x)$ of elements of $V^\pm\setminus \{x,x^{-1}\}$ commuting with $x$. When $x\in V$, these are exactly the standard generators at distance $1$ from $x$ in $\Gamma$ and their inverses, i.e.~the labels of the edges spanning a square with the edge labeled $x$ in $\SS$. A \emph{standard} (sometimes also called \emph{special}) \emph{subgroup} is a subgroup of $A_\Gamma$ generated by some subset $W\subseteq V$. Considering the full subgraph $\Delta$ of $\Gamma$ induced by $W$, its Salvetti complex $\SS_\Delta$ associated with $A_\Delta$ embeds as a locally convex subcomplex of $\SS$, and every locally convex subcomplex of $\SS$ is of this form. In particular, the inclusion $\SS_\Delta\hookrightarrow \SS_\Gamma$ is $\pi_1$-injective, so $A_\Delta = \pi_1\SS_\Delta$ embeds in $\pi_1\SS$ as the subgroup generated by $W$: standard subgroups are right-angled Artin groups themselves. The center of $A_\Gamma$ coincides with the standard subgroup $A_Z$ generated by the set $Z$ of all $v\in V$ such that $lk(v)\cup\{v,v^{-1}\} = V$. The group of inner automorphisms $\inn(A_\Gamma)$ identifies with $A_\Gamma/A_Z\simeq A_{\Gamma\setminus Z}$.

A word in the standard generators and their inverses is \emph{reduced} when it is of minimal length among all words representing the same element of $A_\Gamma$. It is \emph{cyclically reduced} when it is of minimal length among all words representing an element of the same conjugacy class. Any word can be turned into any reduced word representing the same element by a sequence of exchanges of adjacent letters representing commuting elements in $V^\pm$ and deletions of two-letter subwords of the form $vv^{-1}$ or $v^{-1}v$ (\cite{normalformraags}, Proposition~3.2). Additionally, any two reduced words representing the same element differ by a sequence of exchanges of adjacent commuting letters.

We say that a standard generator $v\in V$ \emph{appears} in a word if $v$ or $v^{-1}$ is a (possibly repeated) letter of that word. Since reductions do not introduce new letters, if $W\subset V$, and $w$ is a reduced word in $V$ that represents an element of the standard subgroup generated by $W$, all the standard generators appearing in $w$ belong to $W$.

Finally, when $A_\Gamma$ acts on some metric space $X$ by isometries, for $g\in A_\Gamma$ let $\ell_X(g)=\inf_{x\in X} d_X(x, gx)$ be the \emph{translation length of $g$ in $X$}. It depends only on the conjugacy class of $g$. When $X$ is simply connected and the action is properly discontinuous and cocompact, the isomorphism $A_\Gamma\simeq \pi_1(X/A_\Gamma)$ realizes the translation length of $g$ as the infimum of lengths of loops in the free homotopy class representing the conjugacy class of $g$. In particular, when $X = \USS$ with the combinatorial metric, the quantity $\ell(g)\coloneqq \ell_{\USS}(g)$ is the word length of any cyclically reduced word in the standard generators representing the conjugacy class of $g$.
\end{defi}

The following definition, following \cite{untwistedos}, singles out some particular automorphisms of $A_\Gamma$.

\begin{defi}
We consider partitions $\mathbf{P} = (P,P^*,L)$ of $V^\pm$ into three parts: $V^\pm = P\sqcup P^*\sqcup L$, where only $L$ might be empty. $L$, is the \emph{link} of $\mathbf{P}$, and $P,P^*$ are the \emph{sides} of $\mathbf{P}$. We denote:
\[\begin{aligned}
lk(\mathbf{P}) &= L\\
single(\mathbf{P}) &= \{x\in V^\pm\mid x\in P,\,x^{-1}\in P^* \text{ or }x\in P^*,\,x^{-1}\in P\}\\
double(P) &= \{x\in V^\pm \mid x,x^{-1}\in P\}\\
double(P^*) &= \{x\in V^\pm \mid x,x^{-1}\in P^*\}
\end{aligned}\]

A \emph{based Whitehead partition} $(\mathbf{P},b)$ is the data of a partition $\mathbf{P}$ as above and a \emph{basepoint} $b\in V^\pm$, such that:
\begin{itemize}
\item $b\in P$, $b^{-1}\in P^*$, and $L=lk(b)$.
\item If $x\in P$ and $x^{-1}\in P^*$, $lk(x)\subseteq L$ (we say that $\mathbf{P}$ \emph{splits} $x$).
\item If $x\in P$ and $y\in P^*$ are not inverses, $x$ and $y$ do not commute.
\item $P$ and $P^*$ both contain at least two elements.
\end{itemize}
A \emph{Whitehead partition} is a partition $\mathbf{P}$ such that there exists $b\in V^\pm$ making $(\mathbf{P},b)$ a based Whitehead partition. Whitehead partitions often have several basepoints. A Whitehead partition is entirely determined by any basepoint and one side.

Given $(\mathbf{P},b)$ a based Whitehead partition, the corresponding \emph{Whitehead automorphism} is defined on standard generators as follows:
\[v\mapsto \begin{cases}
v^{-1} &\text{ if }v\in\{b,b^{-1}\}\\
vb^{-1} &\text{ if }v\in single(\mathbf{P})\cap (P\setminus \{b\})\\
bv &\text{ if }v\in single(\mathbf{P})\cap (P^*\setminus \{b^{-1}\})\\
bvb^{-1} &\text{ if }v\in double(P)\\
v &\text{ if }v \in double(P^*)\cup L
\end{cases}\]
While one often finds $b\mapsto b$ in the literature, we adopt here the convention of \cite{untwistedos} which has the advantages of making the automorphism involutive and having a convenient topological realization.
\end{defi}

We recall the following names for some other automorphisms of $A_\Gamma$:

\begin{defi}
An automorphism $\varphi$ of $A_\Gamma$ is called:
\begin{itemize}
\item an \emph{oriented graph permutation} if $\varphi(V^\pm) = V^\pm$. An oriented graph permutation is an \emph{inversion} if it maps one standard generator to its inverse and fixes all the others.
\item a (dominated right-)\emph{transvection} if for some $v_1\neq v_2\in V$, $\varphi(v_1) = v_1v_2$ and $\varphi$ fixes every standard generator different from $v_1$. More precisely it is a \emph{twist} if $v_1$ and $v_2$ commute and a \emph{fold} otherwise. Note that every standard generator in $lk(v_1)$ needs to centralize $v_2$ for this automorphism to be well-defined.

 \item a(n extended) \emph{partial conjugation} if there exists some $v\in V$ such that for every $w\in V$, $\varphi(w)\in \{w,\,vwv^{-1}\}$. Note that any two standard generators commuting with each other but not with $v$ need to be both fixed or both conjugated for this automorphism to be well-defined.
\end{itemize}

All this terminology carries over to outer automorphisms.
\end{defi}

By work of Laurence (\cite{laurence}) and Servatius (\cite{servatius}), oriented graph permutations, transvections and partial conjugations generate all of $\aut(A_\Gamma)$.

\begin{nota}
For $\varphi$ an automorphism of $A_\Gamma$, we will denote by $[\varphi]$ the outer automorphism it represents. Likewise, if $S\subseteq A_\Gamma$, $[S]$ will denote the conjugacy class of $S$. Given $[\varphi]$ an outer automorphism and $S\subseteq A_\Gamma$, the images of $S$ under representatives of $[\varphi]$ form a conjugacy class of subsets that we will denote $[\varphi(S)]$.
\end{nota}

The group $\out(A_\Gamma)$ is virtually torsion-free (\cite{virtuallytorsionfree}), and the right-angled Artin group $\inn(A_\Gamma)$ is torsion-free. Therefore, $\aut(A_\Gamma)$ is virtually torsion-free and so are all the following subgroups of $\aut(A_\Gamma)$ and $\out(A_\Gamma)$:

\begin{defi}
The \emph{untwisted automorphism group} $\uaut(A_\Gamma)<\aut(A_\Gamma)$ is the subgroup generated by inner automorphisms, oriented graph permutations and Whitehead automorphisms (or equivalently oriented graph permutations, folds and partial conjugations). The \emph{untwisted outer automorphism group} $\unt(A_\Gamma)<\out(A_\Gamma)$ is its outer automorphism image.

The \emph{pure automorphism group} $\aut^0(A_\Gamma)<\aut(A_\Gamma)$ is the subgroup generated by inversions, transvections and partial conjugations. It is normal and finite-index. The \emph{pure outer automorphism group} $\out^0(A_\Gamma)<\out(A_\Gamma)$ is its outer automorphism image.

Given $\mathcal{G} = (G_i)$ and $\mathcal{H} = (H_j)$ two families of subgroups of $A_\Gamma$, the \emph{McCool group} $\out(A_\Gamma;\,\mathcal{G},\, \mathcal{H}^t)$ is the subgroup of outer automorphisms $[\varphi]\in \out(A_\Gamma)$ such that for all $i$, $[\varphi(G_i)] = [G_i]$ and for all $j$, there exists $\psi_j \in \inn(A_\Gamma)$ with $(\psi_j)_{\mid H_j} = \varphi_{\mid H_j}$.

The corresponding \emph{untwisted McCool group} and \emph{pure McCool group} are :
\[ \unt(A_\Gamma;\, \mathcal{G},\, \mathcal{H}^t) = \unt(A_\Gamma)\cap \out(A_\Gamma;\,\mathcal{G},\, \mathcal{H}^t)\]
\[\out^0(A_\Gamma;\,\mathcal{G},\, \mathcal{H}^t)=\out^0(A_\Gamma)\cap \out(A_\Gamma;\,\mathcal{G},\, \mathcal{H}^t)\]
\end{defi}

The letter $t$ stands for the fact that the automorphisms act \emph{trivially} up to conjugacy on the subgroups of $\mathcal{H}$. In this article, $\mathcal{G}$ and $\mathcal{H}$ will always be finite collections of finitely generated subgroups. We omit a collection if it is empty.

The following result combines work of Laurence (\cite{laurence}) and Fioravanti (\cite{coarsemedian}).

\begin{pro}
\label{fioravanti}
Let $\varphi\in \aut(A_\Gamma)$. Assume that for every standard generator $v\in V$, $v$ appears in every cyclically reduced word representing $[\varphi(v)]$. Then the following are equivalent:
\begin{enumerate}
\item $\varphi$ is untwisted (i.e.~belongs to $\uaut(A_\Gamma)$)
\item $\varphi$ is \emph{simple}: for every standard generator $v\in V$, the subgraph of $\Gamma$ induced by the set of the standard generators that appear in every cyclically reduced representative of $[\varphi(v)]$ does not split as a nontrivial join
\item $\varphi$ decomposes as a product of inversions, folds and partial conjugations.
\end{enumerate}
\end{pro}
\begin{proof}
$(1.)\Rightarrow(2.)$ by \cite{coarsemedian}, Corollary~3.25, Corollary~3.3, and Lemma~3.11.

$(2.)\Rightarrow(3.)$ by \cite{laurence}, Proposition~6.8, although the decomposition explicitly appears only in \cite{laurence}, proof of Corollary~to Lemma~6.6. It is not explicit in \cite{laurence} but remarked by Fioravanti (\cite{coarsemedian}, proof of Proposition~3.26) that none of the automorphisms of the decomposition are twists.

$(3.)\Rightarrow(1.)$ by definition of the untwisted subgroup.
\end{proof}

Finally, we use the following result of Day and Wade:
\begin{pro}[\cite{relativeout}, from Proposition~3.5]
\label{daywade}
Given $\mathcal{G}$ a family of standard subgroups of $A_\Gamma$, there exists a graph $\widehat{\Gamma}$ containing $\Gamma$ as a full subgraph such that the conjugacy class of $A_\Gamma<A_{\widehat{\Gamma}}$ is invariant under $\out^0(A_{\widehat{\Gamma}})$ and the restriction homomorphism $\out^0(A_{\widehat{\Gamma}})\to \out(A_\Gamma)$ has exactly $\out^0(A_\Gamma;\,\mathcal{G})$ for image.

Moreover, given $\varphi\in \aut^0(A_\Gamma;\,\mathcal{G})$, there exists $\widehat{\varphi}\in \aut^0(A_{\widehat{\Gamma}})$ such that for every vertex $v$ of $\widehat{\Gamma}$, the following holds:
\begin{itemize}
    \item If $v$ is a vertex of $\Gamma$, $\widehat{\varphi}(v) = \varphi(v)$
    \item Otherwise, $[\widehat{\varphi}(v)]=[v]$
\end{itemize}
\end{pro}

From this proposition and the fact that any outer inversion, transvection or partial conjugation in $\out^0(A_{\widehat{\Gamma}})$ restricts to an outer automorphism of $A_\Gamma$ of the same nature, it is clear that $\out^0(A_\Gamma;\,\mathcal{G})$ is generated by the inversions, transvections and partial conjugations it contains. We prove the analogue for the untwisted subgroup.

\begin{cor}
\label{gen}
Given $\mathcal{G}$ a family of standard subgroups of $A_\Gamma$, the group $\unt(A_\Gamma;\,\mathcal{G})$ is generated by the outer oriented graph permutations, outer folds and outer partial conjugations it contains.
\end{cor}
\begin{proof}
Let $[\varphi]\in \unt(A_\Gamma;\,\mathcal{G})$ be represented by $\varphi\in \uaut(A_\Gamma)$. By \cite{laurence}, Corollary~to the Lemma~4.5, there exists a decomposition $\varphi = \psi\circ \sigma$ where $\psi \in \aut(A_\Gamma)$, for every standard generator $v\in V$, $v$ appears in every cyclically reduced word representing $\psi(v)$, and $\sigma$ is a graph permutation (but we don't know yet that $[\psi],[\sigma]\in \unt(A_\Gamma;\,\mathcal{G})$). Let $A_\Delta \in \mathcal{G}$ and $v$ a vertex of $\Delta$. The standard generator $\sigma(v)$ appears in every word representing the conjugacy class $[\varphi(v)]$. Some representative of this class lies in $A_\Delta$, hence $\sigma(v)\in \Delta$. Thus, $[\sigma]\in \unt(A_\Gamma;\,\mathcal{G})$, and $[\psi]\in \unt(A_\Gamma;\,\mathcal{G})$ as well.

Now $[\psi]\in \out^0(A_\Gamma;\,\mathcal{G})$ and $\psi$ is simple by Proposition~\ref{fioravanti}. Using Proposition~\ref{daywade}, there exists an extension $\widehat\psi$ of $\psi$ to $A_{\widehat\Gamma}$, and this extension is simple as well. Thus, by Proposition~\ref{fioravanti} applied to $A_{\widehat\Gamma}$, $[\widehat\psi]$ decomposes as a product of outer inversions, outer folds and outer partial conjugations in $A_{\widehat\Gamma}$. By Proposition~\ref{daywade} again, every factor of the product stabilizes $A_\Gamma$, and restricts to an automorphism of the same type in $\out^0(A_\Gamma;\,\mathcal{G})$, proving the result.
\end{proof}

The following lemma will not be used in the rest of this work, but is a noteworthy consequence of the result of Laurence used in the proof above. It is a relative version of the theorem of Droms (\cite{isomorphismraags}) stating that for every isomorphism of right-angled Artin groups $A_{\Gamma}\simeq A_{\Delta}$, there is an isomorphism of graphs $\Gamma\simeq \Delta$.
\begin{lm}
\label{relativedroms}
    Let $\Gamma,\Delta$ be two finite simplicial graphs. Let $\Gamma_1,\dots, \Gamma_n$ be a family of full subgraphs of $\Gamma$ and $\Delta_1,\dots,\Delta_n$ a family of full subgraphs of $\Delta$. Assume there exists an isomorphism $\varphi\colon A_\Gamma\to A_\Delta$ such that for all $i\leq n$, $[\varphi(A_{\Gamma_i})]=[A_{\Delta_i}]$. Then there exists an isomorphism of graphs $\iota\colon \Gamma\to \Delta$ such that $\iota(\Gamma_i) = \Delta_i$ for all $i\leq n$.
\end{lm}
\begin{proof}
    By the original theorem of Droms, we can assume that $\Gamma = \Delta$, with vertex set $V$, $\varphi$ becoming an automorphism. By \cite{laurence}, Corollary~to the Lemma~4.5, there exists a decomposition $\varphi = \psi\circ \sigma$ where for every standard generator $v\in V$, $v$ appears in every cyclically reduced word representing $\psi(v)$, and $\sigma$ is a graph permutation automorphism. Let $i\leq n$ and $v$ be a vertex of $\Gamma_i$. The standard generator $\sigma(v)$ appears in every word representing the conjugacy class $[\varphi(v)]$. Some representative of this class lies in $A_{\Delta_i}$, hence $\sigma(v)\in \Delta_i$. Since the subgraphs $\Gamma_i$ and $\Delta_i$ are full, this proves that the graph automorphism $\iota\colon \Gamma\to \Gamma$ corresponding to $\sigma$ satisfies $\iota(\Gamma_i) \subseteq \Delta_i$. Since $A_{\Gamma_i}$ and $A_{\Delta_i}$ are isomorphic via a conjugate of $\varphi$, their abelianizations have the same ranks, hence $\Gamma_i$, $\iota(\Gamma_i)$ and $\Delta_i$ have the same number of vertices. This proves that $\iota(\Gamma_i)=\Delta_i$ for all $i\leq n$.
\end{proof}

Finally, the following lemma embeds the untwisted automorphism group of $A_\Gamma$ as a McCool subgroup for the larger right-angled Artin group $A_\Gamma*\ZZ$.

\begin{lm}
\label{autmccool}
    Consider $\Gamma'$ the graph obtained by adjoining to $\Gamma$ an isolated vertex $t$, and $A_{\Gamma'}\simeq A_\Gamma*\ZZ$ the corresponding right-angled Artin group. Let $\mathcal{G} = \{A_\Gamma\}$ and $\mathcal{H} = \{\gen{t}\}$. Then $\uaut(A_\Gamma)\simeq \unt(A_{\Gamma'};\,\mathcal{G},\mathcal{H}^t)$.
\end{lm}
\begin{proof}
    Extend every $\varphi\in \uaut(A_\Gamma)$ as $\widehat{\varphi}\in \aut(A_{\Gamma'})$ by fixing $t$. It is clear from the definition that $\widehat\varphi$ is untwisted, and $\varphi\mapsto [\widehat\varphi]$ defines a group homomorphism from $\uaut(A_\Gamma)$ to $\unt(A_{\Gamma'};\,\mathcal{G},\mathcal{H}^t)$. If $\widehat\varphi$ is inner, it is the conjugacy by some element $g\in A_\Gamma$ centralizing $t$. By \cite{servatius}, Centralizer Theorem, $g$ is a power of $t$. The only conjugacy by a power of $t$ preserving $A_\Gamma$ is the identity. This proves that the homomorphism is injective.

    Every outer oriented graph permutation, outer fold or outer partial conjugation contained in $\unt(A_{\Gamma'};\,\mathcal{G},\mathcal{H}^t)$ is the outer class of a graph permutation, fold or partial conjugation fixing $t$ and preserving $A_\Gamma$. By Corollary~\ref{gen}, the homomorphism is surjective.
\end{proof}

\subsection{Untwisted outer space}

The \emph{spine of untwisted outer space}, $K_\Gamma$, was constructed in \cite{untwistedos}. We will use mostly the equivalent description from \cite{spatial} except to prove a few intermediate results to the main theorem (Lemmas~\ref{whitehead}, \ref{relativepartition} and \ref{pushing}). Below, we recall the main features of both constructions.

\begin{defi}[see \cite{spatial}, Definition~3.1, Lemmas~3.2, 3.7]
\label{defcollapse}
    A \emph{collapse between $\cat$ cube complexes} is a surjective map $c\colon X\to Y$ between $\cat$ cube complexes, that maps cubes of $X$ to cubes of $Y$ affinely, and such that in restriction to each cube, the kernel of $c$ is generated by a subset of the edges. Equivalently, it is a restriction quotient (in the sense of \cite{capracesageev}). A \emph{collapse} is a map of locally $\cat$ cube complexes $c\colon X\to Y$ lifting to a collapse of $\cat$ cube complexes between universal covers $\uc{X}\to \uc{Y}$.
    
    For every hyperplane $H$ of $X$, either all edges dual to $H$ are mapped to vertices ($H$ is \emph{collapsed}), or every edge dual to $H$ is mapped bijectively to an edge. The set of collapsed hyperplanes in $X$ determines $c$ up to isomorphism of the range.
    
    A \emph{strong collapse} is a collapse map $c\colon X\to Y$ satisfying the further assumptions:
    \begin{itemize}
        \item $c$ is a homotopy equivalence.
        \item Any two parallel edges of $X$ mapping to the same vertex of $Y$ are parallel inside the preimage of that vertex (i.e. there exist a sequence of squares witnessing their parallelism that is entirely mapped to that vertex).
    \end{itemize}
\end{defi}

\begin{lm}
\label{collapseupstairs}
    Let $c\colon X\to Y$ be a collapse map between $\cat$ cube complexes. Then $c$ is surjective, preimages of cubes of $Y$ are convex subcomplexes of $X$. Moreover $c$ preserves medians and maps geodesic edge paths to geodesic edge paths (up to reparametrization). The restriction of $c$ to $0$-skeleta is $1$-Lipschitz.

    More generally, if $D\subset Y$ is convex, $c^{-1}(D)\subset X$ is convex, and if $E\subset X$ is convex, $c(E)\subset Y$ is convex,
\end{lm}
\begin{proof}
    The first part of the lemma comes from \cite{spatial}, Lemma~3.2 and 3.9 (the rightmost inequality does not require the collapse to be strong).

    Assume that $D\subset Y$ is convex. Let $(x_1,x_2,x_3)\in c^{-1}(D)\times c^{-1}(D)\times X$ be a triple of vertices. Then $c(\mu(x_1,x_2,x_3)) = \mu(c(x_1),c(x_2),c(x_3))\in D$ since $D$ is convex, hence $\mu(x_1,x_2,x_3)\in c^{-1}(D)$. It remains to see that $c^{-1}(D)$ is full. Let $\mathcal{C}$ be a cube of $X$ all of whose vertices are in $c^{-1}(D)$. Its image under $c$ is a cube all of whose vertices are in $D$. Since $D$ is full, $c(\mathcal{C})$ is contained in $D$, proving that $c^{-1}(D)$ is full.
    
    Likewise, assume $E\subset X$ is convex. Let $(y_1,y_2,y_3)\in c(E)\times c(E)\times Y$ be a triple of points. Since $c$ is surjective, there exists $(x_1,x_2,x_3)\in E\times E\times X$ with $c(x_i) = y_i$. Then $\mu(y_1,y_2,y_3) = c(\mu(x_1,x_2,x_3))\in c(E)$ since $E$ is convex. It remains to see that $c(E)$ is full. Let $\mathcal{C}$ be a cube of $Y$ all of whose vertices are in $c(E)$. The hyperplanes $H_1,\dots, H_n$ of $\mathcal{C}$ are the images of hyperplanes $H'_1,\dots, H'_n$ of $X$. For each $i$, there exist two vertices of $c^{-1}(\mathcal{C})\cap E$ separated by $H'_i$. By convexity of $c^{-1}(\mathcal{C})\cap E$, some edge of $c^{-1}(\mathcal{C})\cap E$ is dual to $H'_i$. More generally, for each $i\neq j$, there exist four vertices of $c^{-1}(\mathcal{C})\cap E$, in each of the four possible intersections of halfspaces for $H'_i$ and $H'_j$. This means that the hyperplanes $H'_i$ are pairwise transverse in the convex $c^{-1}(\mathcal{C})\cap E$. Hence, these hyperplanes are dual to a cube $\mathcal{C}'$ of $c^{-1}(\mathcal{C})\cap E$. Its image $c(\mathcal{C}')$ is a subcube of $\mathcal{C}$ contained in $c(E)$, and dual to all the $H_i$. Therefore, $c(\mathcal{C}') = \mathcal{C}\subseteq c(E)$, and $c(E)$ is full.
\end{proof}
\begin{cor}
\label{collapsedownstairs}
    Let $c\colon X\to Y$ be a collapse map between locally $\cat$ cube complexes and a homotopy equivalence. Let $\gamma$ be an edge cycle in $X$ of minimal length in its (free) homotopy class. Then $c\circ \gamma$ is of minimal length in its homotopy class. Vertex preimages under $c$ are locally convex and $\cat$.
\end{cor}
\begin{proof}
    Lift $c$ to a collapse between the $\cat$ universal covers $\uc{c}\colon \uc{X}\to \uc{Y}$, and lift $\gamma$ to a bi-infinite edge path $\uc{\gamma}$ in $\uc{X}$, which is $[\gamma]$-stable, where $[\gamma]$ is the deck transformation of $\uc{X}$ corresponding to $\gamma$. Since $\gamma$ is of minimal length in its homotopy class, $\uc{\gamma}$ is a geodesic edge path. Thus, by Lemma~\ref{collapseupstairs}, $\uc{c}\circ \uc{\gamma}$ is a geodesic edge path in $\uc{Y}$. Since $c$ is a homotopy equivalence, $\uc{c}\circ \uc{\gamma}$ is $c_*[\gamma]$-invariant. For every vertices $y_1,y_2$ of $Y$, $\card{d(c_*[\gamma]^n(y_1), y_1)/n - d(c_*[\gamma]^n(y_2),y_2)/n}\leq 2d(y_1,y_2)/n\underset{n\to\infty}\to 0$. Thus, all $c_*[\gamma]$-invariant geodesics of $Y$ are translated the same amount by $c_*[\gamma]$. Therefore, $c\circ \gamma$ is of minimal length in its homotopy class.

    Since $c$ is a homotopy equivalence, if the image of $\gamma$ is contained in a vertex preimage under $c$, $\gamma$ is nullhomotopic. Since vertex preimages under $\uc{c}$ are convex, vertex preimages under $c$ are locally convex and $\cat$.
\end{proof}

For the next lemma, recall that a hyperplane in a locally $\cat$ cube complex is \emph{two-sided} when all of its dual edges can be oriented in a way consistent with parallelism.

\begin{lm}[\cite{spatial}, Lemma~3.10]
\label{tubenew}
    Let $c\colon X\to S$ be a strong collapse map, with $X$ compact, locally $\cat$ and $S$ combinatorially isomorphic to $\SS$. Assume all hyperplanes of $X$ are two-sided. Let $C$ be the unique vertex preimage in $X$. Let $e$ be an edge of $X$ outside $C$, with dual hyperplane $H_e$, and let $p$ be a geodesic edge path in $C$ joining the endpoints of $e$. The following hold:
\begin{itemize}
    \item Every edge $f$ spanning a square with $e$ spans a product with the cycle $ep^{-1}$.
    \item Every hyperplane transverse to $H_e$ is transverse to all the hyperplanes dual to edges of $p$.
\end{itemize}
\end{lm}

\begin{cor}
\label{tubemore}
    In the setup of Lemma~\ref{tubenew}, let $e'$ be any edge dual to $H_e$ and $p'$ any geodesic path in $C$ joining the endpoints of $e'$. Then $p$ and $p'$ cross the same hyperplanes.
\end{cor}
\begin{proof}
    The result is clear when $e=e'$ since $C$ is $\cat$. We prove the result when $e$ and $e'$ are opposite edges in a square, the full result follows since $H_e$ is connected. Let $efe'^{-1}f'^{-1}$ be the boundary path of that square. If the edges $f$ and $f'$ are in $C$, then they are parallel in $C$ since $c$ is a strong collapse. In that case, the separator in $C$ of the endpoints of $e$ must be equal to the separator of the endpoints of $e'$, proving the result. Otherwise, let $q$ be a geodesic edge path in $C$ joining the endpoints of $f$. By Lemma~\ref{tubenew}, $e$ spans a product with $fq^{-1}$. The previous case applied successively to each square of the product of $e$ and $q$ proves the result in that case.
\end{proof}

\begin{defi}[see \cite{spatial}, Definition~5.4]
\label{defspatial}
    A \emph{spatial cube complex} is a locally $\cat$ cube complex $X$ satisfying the following assumptions:
    \begin{enumerate}
        \item $X$ is connected and not a single vertex.
        \item $X$ is not obtained by subdivision of some locally $\cat$ cube complex along a hyperplane.
        \item For each hyperplane $H$ of $X$, there exists a strong collapse $c\colon X\to S$ where $H$ is not collapsed and $S$ is combinatorially isomorphic to the Salvetti complex $\SS$.
    \end{enumerate}
    In particular $\pi_1 X \simeq \pi_1 \SS = A_\Gamma$. Note that $\SS$ is spatial. Spatial cube complexes are special (\cite{spatial}, Lemma~5.6).

    A \emph{marking} on a spatial cube complex $X$ is an isomorphism $m\colon \pi_1X\to A_\Gamma$ considered modulo conjugacy (allowing to omit the basepoint of $X$). Markings on $\SS$ are in one-to-one correspondance with elements of $\out(A_\Gamma)$. A marking $m$ on $X$ is \emph{untwisted} if for some (or every, see \cite{untwistedos}, Corollary~4.13) strong collapse $c\colon X\to \SS$, the marking $m\circ c_*^{-1}$ on $\SS$ corresponds to an element of $\unt(A_\Gamma)$. By work of Fioravanti (\cite{coarsemedian}), this is equivalent to the following: the action of $A_\Gamma$ on $\uc{X}$ by deck transformations given by $m$ induces the same coarse median structure on $A_\Gamma$ as the standard action of $A_\Gamma$ on $\SS$.
    
    The \emph{spine of untwisted outer space for $A_\Gamma$} is the simplicial complex $K_\Gamma$ defined as follows:
    \begin{itemize}
        \item Vertices of $K_\Gamma$ are spatial cube complexes $X$ with an untwisted marking $m$, modulo marking-preserving isomorphisms. Their equivalence classes are denoted $[X,m]$.
        \item Edges of $K_\Gamma$ are marking-preserving non-isomorphic collapse maps between vertices.
        \item $k$-cells of $K_\Gamma$ are sequences of $k$ composable edges.
    \end{itemize}
    $\unt(A_\Gamma)$ acts on $K_\Gamma$ by combinatorial isomorphisms, via post-composition with the markings. \emph{Marked Salvettis} are vertices of $K_\Gamma$ of the form $[\SS,\varphi]$ with $\varphi \colon A_\Gamma = \pi_1\SS\to A_\Gamma$, such that $[\varphi]\in \unt(A_\Gamma)$ (the vertex does not depend on the chosen representative for $[\varphi]$). $\unt(A_\Gamma)$ acts transitively on marked Salvettis, with finite stabilizers.
    
    It is clear that a marking-preserving collapse map between spatial cube complexes must be a homotopy equivalence. It is in fact a strong collapse (\cite{spatial}, Corollary~5.8).
    
    Note that since a non-isomorphic collapse map decreases the number of hyperplanes, $K_\Gamma$ is a flag complex: by \cite{spatial}, Lemma~5.11, two vertices are joined by at most one edge. Moreover, given a complete subgraph of its $1$-skeleton, order its vertices by decreasing number of hyperplanes. The edges joining these vertices in order correspond to composable collapses, proving that the subgraph bounds a simplex.

    A \emph{Whitehead move} is an edge path of length $2$ in $K_\Gamma$ from a marked Salvetti to a marked spatial cube complex with two vertices and then to a different marked Salvetti. A \emph{Whitehead path} is a concatenation of Whitehead moves going through each marked Salvetti at most once (this terminology follows \cite{untwistedos}, see Definition~4.18).
\end{defi}

This is not the original definition of $K_\Gamma$ given in \cite{untwistedos}, but it is equivalent to it by \cite{spatial}, Theorem~5.14. The main theorem of \cite{untwistedos}, that the present article generalizes, states that $K_\Gamma$ is contractible.

\begin{defi}[see \cite{untwistedos}, Section~3]
    Two Whitehead partitions $\mathbf{P}$ and $\mathbf{Q}$ of $V^\pm$ are \emph{adjacent} when some (every) pair of basepoints for them are neighbors in $\Gamma$. They are \emph{compatible} if they are adjacent or exactly one of $P\cap Q, P\cap Q^*, P^*\cap Q, P^*\cap Q^*$ is empty (see \cite{twistedos}, Definitions~2.7 and 2.8, amending \cite{untwistedos}, Definition~3.3). For any set $\mathbf{\Pi}$ of pairwise compatible Whitehead partitions, there exists a locally $\cat$ cube complex $\SS^\mathbf{\Pi}$ of called a \emph{blow-up} of the Salvetti complex $\SS$ together with a \emph{canonical collapse} $c_\mathbf{\Pi}\colon \SS^\mathbf{\Pi}\to \SS$.
\end{defi}

By \cite{spatial}, Proposition~5.7, being a spatial cube complex is equivalent to being isomorphic to some blow-up of $\SS$. We will not give the details of the definition of blow-ups (see \cite{untwistedos}) but we recall some useful properties.

\begin{lm}[see \cite{untwistedos}, Section~3]
\label{blowupproperties}
Let $\mathbf{\Pi}$ be a family of pairwise compatible Whitehead partitions. The following hold
\begin{itemize}
        \item $\SS^\mathbf{\Pi}$ is a compact special cube complex with fundamental group $A_\Gamma$ and $c_\mathbf{\Pi}$ is a homotopy equivalence
        \item There is a bijective labelling of the hyperplanes of $\SS^\mathbf{\Pi}$ by $\mathbf{\Pi}\sqcup V$, and a canonical orientation for hyperplanes (we will sometimes speak of the label of an edge meaning the label of its dual hyperplane).
        \item $c_\mathbf{\Pi}$ is a collapse in the sense of the above definition, collapsing exactly the hyperplanes labelled by partitions.
        \item The unique vertex preimage $C_\mathbf{\Pi}$ under $c_\mathbf{\Pi}$ is a $\cat$ cube complex.
        \item A hyperplane labelled by some partition $\mathbf{P}$ has the same transverse hyperplanes as every hyperplane labelled by a basepoint of the partition $\mathbf{P}$. Two hyperplanes labelled by standard generators are transverse if and only if the generators are adjacent in $\Gamma$. Two hyperplanes labelled by partitions are transverse if and only if the partitions are adjacent.
        \item For any $\mathbf{P}=(P,P^*,L)\in \mathbf{\Pi}$, the associated oriented hyperplane determines a partition of $C_\mathbf{\Pi}$ into positive and negative halfspaces. For any $v\in V$, if $v\in P$ (resp. $v^{-1}\in P$), all edges labelled $v$ of $X$ have their terminal (resp. initial) endpoint in the positive halfspace. Likewise, if $v\in P^*$ (resp. $v^{-1}\in P^*$), all edges labelled $v$ of $X$ have their terminal (resp. initial) endpoint in the negative halfspace.
        \item As a consequence, for any edge $e$ labelled by a standard generator $v$, the set of hyperplanes separating the endpoints of $e$ in $C_\mathbf{\Pi}$ is exactly the set of hyperplanes labelled by partitions $\mathbf{P}$ with $v\in single(\mathbf{P})$.
    \end{itemize}
\end{lm}

\begin{rem}
Since there are only finitely many Whitehead partitions of $V^\pm$ and compatible partitions are distinct, there are finitely many blow-ups of $\SS$. Hence there are finitely many isomorphism types of spatial cube complexes (with a given fundamental group). This means that $K_\Gamma$ has finitely many $\unt(A_\Gamma)$-orbits of vertices, i.e., $\unt(A_\Gamma)$ acts cocompactly on $K_\Gamma$.
    
If $(\mathbf{P},v)$ is a based Whitehead partition, the collapse $c_v\colon \SS^\mathbf{P}\to \SS^\mathbf{P}_v$ from the blow-up $\SS^\mathbf{P}$ where the only collapsed hyperplane is labelled $v$ is a homotopy equivalence. Moreover, there exists an isomorphism $\iota$ between the collapse $\SS^\mathbf{P}_v$ and $\SS$, such that the outer automorphism $[(c_\mathbf{P})_*\circ (c_v)_*^{-1}\circ \iota^{-1}_*]\in\unt(A_\Gamma)$ is the outer Whitehead automorphism corresponding to $(\mathbf{P},v)$. This means that the marked Salvettis one Whitehead move away from $[\SS,id_{A_\Gamma}]$ are exactly the marked Salvettis one of whose markings is a Whitehead automorphism.
\end{rem}

\section{Minsets for cospecial actions}
\label{minsets}

The main goal of this section is to prove the following proposition.

\begin{pro}
\label{bounded_displacement}
Let $G$ be a group acting freely and cospecially on a $\rm{CAT}(0)$ cube complex $X$ of finite dimension $n$ by combinatorial isometries. Recall that for $g\in G$, $\ell(g)$ is the combinatorial translation length of $g$ in $X$. Let $(a_1,\dots , a_k)\in G^k$ and set the following. \[M \coloneqq \max_i \ell(a_i) + \frac{n}{2}\max_{i<j} \ell(a_ia_j)+\frac{3n}{2}\]
There exists a vertex $x\in X^{(0)}$ such that for every $i$, $d(x,a_ix)\leq M$.
\end{pro}

The proof involves several results about invariant subsets in actions of finitely generated groups on $\cat$ cube complexes. Although most of them belong to folklore, we provide complete proofs. Recall that two convex subcomplexes are called parallel when the sets of hyperplanes dual to their edges are identical.
\begin{lm}
\label{invariantminimal}
Let $K$ be a finitely generated group acting on a $\cat$ cube complex $X$ by combinatorial isometries. There exists a non-empty convex subcomplex $Y$ of $X$ that is $K$-invariant and minimal (for the inclusion) with these properties. Besides, $Y$ has finitely many $K$-orbits of hyperplanes. Moreover, any two minimal subcomplexes with these properties are parallel.
\end{lm}
Note that when $X$ has a $K$-invariant convex subcomplex $S$, $Y$ can be taken contained in $S$ simply by applying the result to $S$ instead of $X$.
\begin{proof}
    Let $\Sigma$ be a finite symmetric generating set for $K$ and $x$ a vertex of $X$. The combinatorial convex hull $C$ of $K\cdot x$ (that is, the intersection of all halfspaces containing $K\cdot x$) is convex, non-empty and $K$-invariant. Besides, if a hyperplane $H$ is dual to an edge of $C$, $K\cdot x$ is not contained in a halfspace of $H$ by definition of $C$. Thus, there exists $k_1,k_2\in K$ such that $H\in \s(k_1x\mid k_2x)$. Write $k_1^{-1}k_2 = s_1\dots s_n$ with $s_i\in \Sigma$. The following holds:
    \[\s(k_1x\mid k_2x)\subseteq \bigcup_{i=0}^{n-1} \s(k_1s_1\dots s_ix\mid k_1s_1\dots s_{i+1}x) = \bigcup_{i=0}^{n-1} k_1s_1\dots s_i\cdot \s(x\mid s_{i+1}x)\]
    
    Thus, $H$ is in the same $K$-orbit as some element of $\s(x\mid sx)$ for some $s\in \Sigma$. Since $\Sigma$ and separators are finite, $C$ has finitely many $K$-orbits of hyperplanes.

    Let $Y$ be a $K$-invariant non-empty convex subcomplex of $C$ with a minimal number of $K$-orbits of hyperplanes. Let $Z\subseteq Y$ be a $K$-invariant and non-empty convex subcomplex. By definition of $Y$, $Y$ and $Z$ have the same $K$-orbits of hyperplanes. Assuming $y$ is a vertex of $Y\setminus Z$, let $H\in \s(y\mid Z)$. Then by convexity $H$ is a hyperplane of $Y$, but $Z$ is contained in one of its halfspaces: $H$ is not a hyperplane of $Z$. By $K$-invariance of $Z$, no hyperplane in the $K$-orbit of $H$ is in $Z$, a contradiction. Therefore, $Y$ and $Z$ have the same $0$-skeleton. Since they are convex, hence full, $Z = Y$ is minimal for the inclusion.

    Let $Y'$ be another minimal $K$-invariant subcomplex, and let $(y,y') \in Y\times Y'$ be a pair of vertices at minimal distance. By minimality of $Y$, $Y$ is the combinatorial convex hull of $K\cdot y$. Therefore, for any hyperplane $H$ dual to an edge of $Y$, $H\in \s(y\mid ky)$ for some $k\in K$. Assuming $H\notin \s(y'\mid ky')$, then either $H\in \s(y\mid ky, y',ky')$ or $H\in \s(ky\mid y, y',ky')$. In the first case, $\mu(y,y',ky)\in Y$ belongs to a geodesic joining $y$ and $y'$ yet is separated from $y$ by $H$, contradicting the minimality of $d(y,y')$. Likewise, in the second case, $\mu(k^{-1}y,y,y')\in Y$ belongs to a geodesic joining $y$ and $y'$ yet is separated from $y$ by $k^{-1}H$, yielding the same contradiction. Thus $H\in \s(y'\mid ky')$ is dual to an edge of $Y'$. By symmetry of the assumptions, $Y$ and $Y'$ are parallel.
\end{proof}

In the case of trees ($1$-dimensional $\cat$ cube complexes), parallelism is equality (except for singletons), and one recovers the uniqueness of the minimal invariant subtree. In the following lemma, one considers a free cospecial action of $\ZZ$. Such actions arise naturally when considering a larger group $G$ acting properly discontinuously and cospecially on a $\cat$ cube complex and restricting the action to the subgroup $\gen{g}\simeq \ZZ$ generated by any infinite-order element $g\in G$.
\begin{lm}
\label{minset}
Let $\ZZ = \gen{g}$ act freely and cospecially on a $\cat$ cube complex $X$ by combinatorial isometries. The set of vertices $x\in X^{(0)}$ such that $d(x,gx)$ is minimal spans a non-empty $\gen{g}$-invariant convex subcomplex of $X$. We call this subcomplex the \emph{combinatorial minset} of $g$ and denote it $\mathrm{Min}(g)$.
\end{lm}

When $X$ is a tree, the combinatorial minset coincides with the \emph{axis} of $g$. In general, the vertex set of $\mathrm{Min}(g)$ is the union of vertex sets of all minimal $\gen{g}$-invariant convex subcomplexes of $X$. It is also the union of vertex sets of all \emph{combinatorial axes} of $g$ (i.e.~bi-infinite, $\ell^1$-geodesic edge paths stabilized by $g$).

Note that while Lemma~\ref{invariantminimal} applies for any isometric action on a $\cat$ cube complex, Lemma~\ref{minset} fails without the specialness assumption: vertices minimally translated by a diagonal glide reflection in the standard cubulation of $\RR^2$ lie in a neighborhood of a diagonal line, which can never span a convex subcomplex. The minimal invariant convex subset for $\ZZ$ in that case is the whole space $\RR^2$.

\begin{proof}

It is clear that $\mathrm{Min}(g)$ is non-empty and $\gen{g}$-invariant. By specialness, every hyperplane of $X$ descends to a two-sided hyperplane in the quotient $X/\ZZ$, hence $g$ acts \emph{stably without inversion}, in the sense of \cite{semisimple}. By \cite{semisimple}, Corollary~6.2, the translation length of $g$ is positive, and every vertex in $\mathrm{Min}(g)$ belongs to a combinatorial axis.

Let $x,y$ be vertices in $\mathrm{Min}(g)$ and let $\gamma$ be any combinatorial geodesic path from $x$ to $y$. To prove inductively that $\gamma$ is contained in $\mathrm{Min}(g)$, let $[x,x']$ be the first edge of $\gamma$ and $H$ its dual hyperplane. The following equality holds for $d$ the combinatorial metric:
\[
\begin{aligned}
d(x',gx') &= \card{\s(x'\mid gx')} \\
&= \card{\s(x'\mid x)\sd \s(x\mid gx) \sd \s(gx\mid gx')}\\
&= \card{\s(x\mid gx)\sd \{H\} \sd \{gH\}}\end{aligned}\]

If either $H\in \s(x\mid gx)$, $gH\in \s(x\mid gx)$ or $H=gH$, then $d(x',gx') \leq d(x,gx)$. Hence $x'\in \mathrm{Min}(g)$, and the induction proceeds.

Otherwise, $H\neq gH$ and $\{H,gH\}\cap \s(x\mid gx) = \emptyset$. By specialness, $H$ and $gH$ are not transverse. Hence $H\notin \s(gx\mid gy)$, or else the quadruple $x, x', gy, gx'$ would witness the transversality of $H$ and $gH$. As $H\in \s(x\mid y)$ yet $H\notin \s(x\mid gx)\cup \s(gx\mid gy)$, $H\in\s(y\mid gy)$. Symmetrically, $gH \in \s(y\mid gy)$, which rewrites as $H\in \s(g^{-1}y\mid y)$. Thus a combinatorial axis going through $y$ meets $H$ twice between $g^{-1}y$ and $gy$, yielding a contradiction.
\end{proof}

\begin{lm}
\label{axes}
In the same setup as Lemma~\ref{minset}, for every vertices $x,y\in \mathrm{Min}(g)$, there exists a unique bijection between $\s(x\mid gx)$ and $\s(y\mid gy)$ which maps each hyperplane to a hyperplane in the same $\gen{g}$-orbit. 
\end{lm}

\begin{proof}
Let $H\in\s(x\mid gx)$ and $\omega$ its $\gen{g}$-orbit. Since some combinatorial axis goes through $x$, $\omega\cap \s(x\mid gx) = \{H\}$. For the same reason, $\omega\cap \s(y\mid gy)$ has at most one element.

However $\omega \cap \s(gx\mid gy)$ is a translate of $\omega\cap \s(x\mid y)$, hence of the same cardinality. Thus the cardinality of $\omega \cap \s(y\mid gy) = \omega \cap (\s(y\mid x) \sd \s(x\mid gx)\sd \s(gx\mid gy))$ is odd, hence is exactly one. Pick this unique element as the image of $H$. Since no choices have been made, the obtained bijection is unique.
\end{proof}

\begin{lm}
\label{geodesic}
In the same setup as Lemma~\ref{minset}, let $x\in X^{(0)}$ be a vertex, and $p$ its unique projection on $\mathrm{Min}(g)$. There exists a geodesic path from $x$ to $gx$ which passes through the vertices $p$ and $gp$.
\end{lm}

This isn't necessarily verified by every geodesic, as shown by the universal cover of two squares glued together at two adjacent vertices (but not along the edge). Note that as a consequence of this lemma, every $g$-invariant convex subcomplex intersects $\mathrm{Min}(g)$.

\begin{proof}
The goal is to prove that $\s(x\mid p)$, $\s(p\mid gp)$ and $\s(gp\mid gx)$ are pairwise disjoint.

The separator $\s(x,gp\mid p)$ is empty, otherwise the median $\mu(x,p,gp)$, which belongs to $\mathrm{Min}(g)$ by convexity, would be a better projection than $p$. Symmetrically, $\s(p, gx\mid gp)$ is empty. Hence $\s(x\mid p) = \s(x \mid p,gp)$ and $\s(gx\mid gp) = \s(gx \mid p,gp)$. This proves that $\s(p\mid gp)$ intersects trivially the two others, and that the third intersection $\s(x\mid p)\cap \s(gx\mid gp)$ equals $\s(x,gx\mid p,gp)$.

Assume there exists $H\in \s(x,gx\mid p,gp)$, taken closest to $p$. Recall that $\s(H\mid p)$ denotes the set of hyperplanes separating the carrier of $H$ from $p$. Then $\s(H\mid p)\subseteq \s(x,gx\mid p)$. Moreover, $\s(x,gx\mid p)\cap \s(gp\mid p)\subseteq \s(x,gp\mid p)$, and the latter is empty by the argument above. Thus, $\s(H\mid p)\subseteq \s(x,gx\mid p,gp)$ and any hyperplane $K\in \s(H\mid p)$ would contradict the choice of $H$. This proves that $\s(H\mid p)$ is empty. Hence $p$ belongs to the carrier of $H$ by convexity of the latter. Therefore, $H$ is dual to some edge $[p,p']$ starting at $p$.

As $H$ separates $p$ from $x$, the vertex $p'$ cannot belong to $\mathrm{Min}(g)$. This yields a contradiction if $gH = H$, because then: \[d(p',gp')=\card{\s(p',gp')} =\card{\s(p\mid gp)\sd \{H\}\sd \{gH\}} = d(p,gp)\]

Otherwise, $gH\neq H$, but the quadruple $p', gx, gp', gp$ witnesses the transversality of $H$ and $gH$ which contradicts the specialness assumption.
\end{proof}

\begin{lm}
\label{collapseminset}
    Let $c\colon X\to Y$ be a collapse between $\cat$ cube complexes. Assume $\ZZ=\gen{g}$ acts freely and cospecially on $X$ and $Y$, and $c$ is equivariant. Then $c(\mathrm{Min}_X(g)) = \mathrm{Min}_Y(g)$.
\end{lm}
\begin{proof}
    Let $x\in \mathrm{Min}_X(g)$ be a vertex. There exists a $g$-invariant bi-infinite geodesic edge path of $X$ containing $x$. By Lemma~\ref{collapseupstairs}, there exists a $g$-invariant bi-infinite geodesic edge path of $Y$ containing $c(x)$, hence $c(x)\in \mathrm{Min}_Y(g)$. By convexity, $c(\mathrm{Min}_X(g))\subseteq \mathrm{Min}_Y(g)$.
    
    Conversely, let $y\in \mathrm{Min}_Y(g)$ be a vertex and, using surjectivity of $c$, let $x\in X$ be a vertex in the preimage of $y$. Finally, let $p$ be the projection of $x$ onto the convex subcomplex $\mathrm{Min}_X(g)$. By Lemma~\ref{geodesic}, some geodesic joining $x$ and $gx$ goes through $p$ and $gp$. By Lemma~\ref{collapseupstairs} again, there is a geodesic joining $c(x) = y$ and $gy$ going through $c(p)$ and $gc(p)$. Thus, $d(c(p),gc(p))\leq d(y,gy)$ and equality holds since $y\in \mathrm{Min}_Y(g)$. Therefore, $y=c(p)\in c(\mathrm{Min}_X(g))$, and the result follows since $c(\mathrm{Min}_X(g))$ is convex by Lemma~\ref{collapseupstairs}.
\end{proof}

A well-known result states in the case of trees that if two group elements $g,h$ act with disjoint axes, $\ell(gh) = \ell(g) + \ell(h) + 2d(\mathrm{Min}(g),\mathrm{Min}(h))$. Unfortunately, this fails for general $\cat$ cube complexes. Consider for example the special cube complex $Y$ obtained by attaching two edges to a square, making all four attaching points distinct and one of the hyperplanes separating. Equality does not hold for the universal cover of $Y$ with isometries corresponding to the two edge loops of length $2$ of $Y$, suitably oriented. However, we can still recover an inequality.

\begin{pro}
\label{distance_minsets}
Let $G$ be a group acting freely and cospecially on a $\rm{CAT}(0)$ cube complex $X$ by combinatorial isometries. Let $g,h\in G$. The following inequality holds:
\[d(\mathrm{Min}(g), \mathrm{Min}(h))\leq \frac{\ell(gh)}{2}\].
\end{pro}

\begin{proof}
Let $\Omega$ be the set of $G$-orbits of hyperplanes of $X$, and $\omega$ any element of $\Omega$. By Lemma~\ref{axes}, the number $M_\omega$ of elements of $\omega$ in $\s(x\mid ghx)$ does not depend on the choice of a vertex $x\in \mathrm{Min}(gh)$. Let $m_\omega$ be the number of elements of $\omega$ in $\s(\mathrm{Min}(g)\mid \mathrm{Min}(h))$. By convexity of minsets, the following equalities hold:
\[\sum_{\omega\in \Omega} m_\omega = d(\mathrm{Min}(g), \mathrm{Min}(h))\qquad \sum_{\omega\in \Omega}M_\omega = \ell(gh)
\]
It now suffices to prove $2m_\omega\leq M_\omega$ for each $\omega\in\Omega$ to conclude.

Elements of $\omega$ are pairwise non-transverse by specialness. The collapse $c_\omega\colon X \to T_\omega$ of all hyperplanes of $X$ not contained in $\omega$ has a tree $T_\omega$ for range. Since the family of collapsed hyperplanes is $G$-invariant, there is a natural action of $G$ on $T_\omega$ by combinatorial isometries making $c_\omega$ equivariant. Note that $T_\omega$ is not necessarily locally finite, even when $X$ is. However, $M_\omega=\ell_{T_\omega}(gh)$ and $m_\omega=d(c_\omega(\mathrm{Min}_X(g)), c_\omega(\mathrm{Min}_X(h))) = d(\mathrm{Min}_{T_\omega}(g), \mathrm{Min}_{T_\omega}(h))$ using Lemma~\ref{collapseminset}. If $m_\omega = 0$ the result is obvious, and otherwise $\mathrm{Min}_{T_\omega}(g)$ and $\mathrm{Min}_{T_\omega}(h)$ are disjoint, and the equality for trees: $\ell_{T_\omega}(gh) = \ell_{T_\omega}(g) + \ell_{T_\omega}(h) + 2d(\mathrm{Min}_{T_\omega}(g), \mathrm{Min}_{T_\omega}(h))$ concludes the proof.
\end{proof}

In the following lemma, $\mathrm{N}_r(C)$ denotes the closed $r$-neighborhood of $C$ for the combinatorial metric.
\begin{lm}
\label{convex_neighborhood}
Let $C$ be a convex subcomplex of a $\cat$ cube complex $X$ of finite dimension $n$. Let $r$ be a nonnegative integer. The combinatorial convex hull of $\mathrm{N}_r(C)$ is contained in $\mathrm{N}_{nr}(C)$.
\end{lm}

The bound is optimal, as seen with $X=[0,r]^n$, $C=\{0\}^n$.

\begin{proof}
Let $x$ be a vertex in the combinatorial convex hull of $\mathrm{N}_r(C)$. By convexity of $C$, the following holds:
\[\s(x\mid C) = \bigcup_{\substack{y\in X^{(0)}\\ d(y,C)\leq r}} \s(x,y\mid C)\]

Each term $\s(x,y\mid C)$ of the union is contained in $\s(y\mid C)$, hence has cardinality at most $r$. Besides, since $\s(x\mid C)$ is finite, the index set of the union can be restricted to a finite subset without changing the result.

\emph{Claim:} If this index subset has more than $n$ elements, one of them can be removed without changing the result.

This allows to remove terms in the union until no more than $n$ are left, yielding the desired bound $d(x,C)=\card{\s(x\mid C)}\leq nr$.

\emph{Proof of the claim:} Let $\s(x,y_i\mid C)_{1\leq i\leq n+1}$ be terms of the union. Assume by contradiction that no one is contained in the union of the others. For each $i$, there exists a hyperplane $H_i\in \s(x, y_i\mid C)$ not inside any of the $\s(x, y_j\mid C)_{j\neq i}$. Hence, for each pair $i\neq j$, the quadruple $x, y_i, y_j, C$ witnesses transversality of $H_i$ and $H_j$, but having $n+1$ pairwise transverse hyperplanes is a contradiction in a dimension $n$ complex (their carriers intersect globally by the Helly property).
\end{proof}

We can now finally prove the result.
\begin{proof}[Proof of Proposition~\ref{bounded_displacement}]
Let $M' = \ceil{\max_{i<j} \ell(a_ia_j)/4}\leq \max_{i<j} \ell(a_ia_j)/4+3/4$. For every $i$, set $N_i = \mathrm{N}_{M'}(\mathrm{Min}(a_i))$ and $C_i$ the combinatorial convex hull of $N_i$.

By Proposition~\ref{distance_minsets}, $N_i$ and $N_j$ intersect for every pair $i\neq j$. Hence the collection $C_i$ has pairwise nonempty intersection. By the Helly property, there exists a vertex $x\in X^{(0)}$ which belongs to every $C_i$. By Lemma~\ref{convex_neighborhood}, for every $i$, $d(x,\mathrm{Min}(a_i))\leq nM'$. Let $p_i$ be the unique projection of $x$ onto $\mathrm{Min}(a_i)$. The following holds for every $i$:
\[\begin{aligned}
d(x,a_ix) &\leq d(x,p_i) + d(p_i,a_ip_i) + d(a_ip_i, a_ix)\\
&\leq nM' + \ell(a_i) + nM'\\
&\leq M
\end{aligned}\]
\end{proof}

We end with two lemmas giving a partial converse to Lemma~\ref{invariantminimal} in the case of a cospecial action.
\begin{lm}
\label{parallelseparator}
    Let $A_1,A_2$ be parallel convex subcomplexes of a $\cat$ cube complex $X$. For every vertex $a$ of $A_1$, $\s(a\mid \pi_{A_2}(a))=\s(A_1\mid A_2)$. For every vertices $a,b$ of $A_1$, $\s(a\mid b) = \s(\pi_{A_2}(a)\mid \pi_{A_2}(b))$.
\end{lm}
\begin{proof}
    The inclusion $\s(a\mid \pi_{A_2}(a))\supseteq \s(A_1\mid A_2)$ is clear since $a\in A_1$ and $\pi_{A_2}(a)\in A_2$. Moreover, by convexity of $A_2$, $\s(a\mid \pi_{A_2}(a))$ contains no hyperplane dual to an edge of $A_2$ (otherwise $\mu(a,x,\pi_{A_2}(a))$ would be a better projection, for $x$ the appropriate vertex of such an edge). By parallelism, $\s(a\mid \pi_{A_2}(a))$ contains no hyperplane dual to an edge of $A_1$. Finally, letting $(p,q)\in A_1\times A_2$ be vertices at minimal distance, $\s(a\mid p)$ contains only hyperplanes of $A_1$ and $\s(q\mid \pi_{A_2}(a))$ contains only hyperplanes of $A_2$. Hence, $\s(a\mid \pi_{A_2}(a))\subseteq \s(p\mid q) = \s(A_1\mid A_2)$, proving the equality.

    Now given $a,b\in A_1$:
    \[\begin{aligned}\s(\pi_{A_2}(a)\mid \pi_{A_2}(b)) &= \s(\pi_{A_2}(a)\mid a)\sd \s(a\mid b)\sd \s(b\mid\pi_{A_2}(b))\\
    &=\s(A_1\mid A_2)\sd \s(a\mid b)\sd \s(A_1\mid A_2)\\
    &=\s(a\mid b)
    \end{aligned}\]
\end{proof}

\begin{lm}
\label{parallelstabilizer}
    Let $G$ act freely and cospecially on a $\cat$ cube complex $X$ by combinatorial isometries. Let $A_1,A_2$ be parallel convex subcomplexes of $X$. Assume that for every $g\in G$, $i\in\{1,2\}$, either $gA_i = A_i$ or $gA_i$ and $A_i$ are disjoint. Then $Stab(A_1) = Stab(A_2)$.
\end{lm}
This fails without the specialness assumption in the following example: consider $X=\RR\times [-1,1]$ with its cubical subdivision with vertices at integer points, and let $\ZZ$ act on $X$ by $(x,y)\mapsto (x+1,-y)$. The stabilizer of $\RR\times \{1\}$ has index $2$ in the stabilizer of $\RR\times \{0\}$.

\begin{proof}
Let $g\in Stab(A_1)$ and let $x\in A_1$ be any vertex. Let $p$ be a geodesic edge path joining $x$ to $gx$, and $q$ a geodesic edge path joining $x$ to $\pi_{A_2}(x)$. By Lemma~\ref{parallelseparator}, every hyperplane crossed by $q$ separates $A_1$ from $A_2$. Besides, by parallelism, every hyperplane crossed by $p$ is dual to an edge of $A_1$ and an edge of $A_2$. Therefore, every hyperplane crossed by $p$ is transverse by every hyperplane crossed by $q$. By an inductive application of \cite{richestoraags}, Lemma~3.6, $p$ and $q$ span a product in $X$. Let $q'$ be the path parallel to $p$ on the opposite side of the product starting at $gx$. Since $q'$ crosses the same hyperplanes as $q$, which are exactly the hyperplanes separating $gx$ from $\pi_{A_2}(gx)$ (by Lemma~\ref{parallelseparator} again), $q'$ joins $gx$ to $\pi_{A_2}(gx)$.

Now the paths $q'$ and $gq$ have the same initial endpoint $gx$, and the $n$th (oriented) edges of $q'$ and $gq$ have dual (oriented) hyperplanes in the same $g$-orbit for all $n$. Inductively, by specialness (more specifically, no direct self-osculation), the vertices of $q'$ and $gq$ all coincide. In particular, $\pi_{A_2}(gx) = g\pi_{A_2}(x)$. Therefore, $gA_2$ intersects $A_2$, and thus, by assumption, $g\in Stab(A_2)$. This proves that $Stab(A_1)\subseteq Stab(A_2)$ and the reverse inclusion holds symmetrically.
\end{proof}

\section{Collapses of locally convex subcomplexes}
\label{collapses}

This section presents two useful lemmas about the action of homotopy equivalent collapse maps on locally convex subcomplexes of the base.

\begin{lm}
\label{liftloop}
Let $c\colon X\to X'$ be a collapse map between locally $\cat$ cube complexes, which is a homotopy equivalence. Let $D$ be a connected, locally convex subcomplex of $X$. Choose $*$ any vertex of $D$. The following are equivalent:
\begin{enumerate}
\item The pointed map $c_{\mid D}\colon (D,*)\to (c(D),c(*))$ induces a surjection of fundamental groups.
\item For every vertex preimage $C$ in $X$, $D\cap C$ is connected
\item For every vertex preimage $C$ in $X$, $D\cap C$ is convex in $C$
\item For every deck transformation $g$ of the universal cover $\uc{X}$, either $g\uc{D} = \uc{D}$ or $\uc{c}(g\uc{D})$ is disjoint from $\uc{c}(\uc{D})$, where $\uc{c}\colon\uc{X}\to \uc{X'}$ lifts $c$ between universal covers and $\uc{D}$ is any connected component of the preimage of $D$ in $\uc{X}$. 
\end{enumerate}
Moreover, when they hold, $c(D)$ is a connected, locally convex subcomplex of $X'$, and $c_{\mid D}$ is a homotopy equivalence between $D$ and $c(D)$.
\end{lm}

Note that the intersection $D\cap C$ may be empty, in which case it is vacuously convex and connected.
\begin{proof}
First note that $c_{\mid D}$ is always $\pi_1$-injective. Indeed, (up to restriction of the range) it is the composition of the homotopy equivalence $c$ and the inclusion $D\hookrightarrow C$, the latter being $\pi_1$-injective by local convexity of $D$.

$(4.)\Rightarrow (3.)$ Let $C$ be a vertex preimage in $X$ and $\uc{C}$ a connected component of its preimage in $\uc{X}$. Fix also $\uc{D}$ as in $(4.)$. Since $C$ is locally convex and $\cat$ (Corollary~\ref{collapsedownstairs}), the universal cover restricts to an isomorphism $\uc{C}\to C$, and $\uc{C}$ is convex. Let $g,g'$ be deck transformations and assume that both $g\uc{D}$ and $g'\uc{D}$ intersect $\uc{C}$. Then $\uc{c}(g\uc{D})$ and $\uc{c}(g'\uc{D})$ have in common the vertex $\uc{c}(\uc{C})$. By assumption, $g\uc{D} = g'\uc{D}$. Thus $\uc{C}$ intersects at most one connected component of the preimage of $D$. This component is convex by local convexity of $D$, and the convex intersection is exactly the preimage of $D\cap C$ under the isomorphism $\uc{C}\to C$.

$(3.)\Rightarrow (2.)$ is clear.

$(2.)\Rightarrow (1.)$ Let $\gamma$ be a closed edge path in $c(D)$ based at $c(*)$. Pick, for each edge of $\gamma$, an edge of $D$ of which it is the image. For each vertex of $\gamma$ different from $*$, with preimage $C$ in $X$, the incoming and outgoing edges of $\gamma$ at this vertex correspond to chosen edges of $D$, each having an endpoint in $C\cap D$. By assumption, these endpoints can be joined by an edge path in $C\cap D$. Likewise at $*$, the first endpoint of the edge of $D$ chosen for the first edge of $\gamma$ and the last endpoint of the edge of $D$ chosen for the last edge of $\gamma$ can both be joined to $*$ by edge paths in $c^{-1}(c(*))\cap D$. This provides a closed edge path in $D$ based at $*$ projecting to $\gamma$ up to reparametrization.

$(1.)\Rightarrow (4.)$ Choose a lift $\uc{*}\in \uc{D}$ of $*$ and a deck transformation $g$ of $\uc{X}$. Assume that $\uc{c}(g\uc{D})$ and $\uc{c}(\uc{D})$ intersect at a vertex $x$. Then there exists a path $p'$ joining $\uc{c}(\uc{*})$ to $\uc{c}(g\uc{*})$, going through $x$, such that the image $\gamma'$ of $p'$ in $X'$ is contained in $c(D)$. The closed path $\gamma'$ represents a homotopy class in $\pi_1(c(D),c(*))$. By assumption, there exists a closed path $\gamma$ in $D$ based at $*$ such that $c\circ \gamma$ is homotopic to $\gamma'$. Lift $\gamma$ to a path $p$ in $\uc{X}$ starting at $\uc{*}$, and ending at some $h\uc{*}$ for $h$ a deck transformation of $\uc{X}$ representing $[\gamma]\in \pi_1(D,*)$. Then $p$ is in $\uc{D}$ and $\uc{c}\circ p$ starts at $\uc{c}(\uc{*})$ and ends at $\uc{c}(g\uc{*})$ like $p'$. Thus $\uc{c}(g\uc{*}) = \uc{c}(h\uc{*})$, and since $c$ is a homotopy equivalence, $g$ and $h$ are the same deck transformation. Therefore, $g=h$ represents an element of $\pi_1(D,*)$, hence $g\uc{D} =\uc{D}$.

This proves the equivalence.
\medskip

Since $D$ is connected, $c(D)$ is connected as well. The preimage of $c(D)$ in the universal cover $\uc{X'}$ is the union over all deck transformations $g$ of $\uc{X}$ of $\uc{c}(g\uc{D})$. Assume $(4.)$ holds. Then $\uc{c}(\uc{D})$ is a connected component of this union, and $\uc{c}(\uc{D})$ is convex by convexity of $\uc{D}$ and Lemma~\ref{collapseupstairs}. Therefore, $c(D)$ is locally convex in $X'$. In particular, $c(D)$ is locally $\cat$, hence both $D$ and $c(D)$ are aspherical, and $c_{\mid D}$ is $\pi_1$-bijective by $(1.)$. This proves that $c_{\mid D}$ is a homotopy equivalence.
\end{proof}

The following lemma makes the further assumption that $X'$ has a single vertex.

\begin{lm}
\label{parallel}
Let $c\colon X\to X'$ be a collapse map between locally $\cat$ cube complexes, which is a homotopy equivalence. Assume $X'$ has only one vertex. Let $D_1,D_2$ be connected, locally convex subcomplexes of $X$. Assume there exists a choice of components $\uc{D_1},\uc{D_2}$ of the preimages of $D_1$ and $D_2$ in the universal cover $\uc{X}$ satisfying the following conditions:
\begin{enumerate}
    \item $\uc{D_1}$ and $\uc{D_2}$ are parallel
    \item $\uc{D_1}$ and $\uc{D_2}$ have the same stabilizer in the group of deck transformations of $\uc{X}$ (by Lemma~\ref{parallelstabilizer}, this follows from $(1.)$ when $X$ is special).
\end{enumerate}
Then $D_1$ satisfies the equivalent conditions from Lemma~\ref{liftloop} if and only if $D_2$ does. Moreover, when $D_1$ and $D_2$ satisfy these conditions and $X'$ is isomorphic to $\SS$, $c(D_1)=c(D_2)$.
\end{lm}

\begin{proof}
Let $\uc{c}\colon \uc{X}\to \uc{X'}$ lift $c$ between universal covers, and assume that for every deck transformation $g$ of $\uc{X}$, $g\uc{D_1} = \uc{D_1}$ or $\uc{c}(g\uc{D_1})$ and $\uc{c}(\uc{D_1})$ are disjoint (i.e.~$D_1$ satisfies Assumption~$(4.)$ of Lemma~\ref{liftloop}). Now let $g$ be a fixed deck transformation of $\uc{X}$ and assume $\uc{c}(g\uc{D_2})$ and $\uc{c}(\uc{D_2})$ intersect. Then there exist vertices $x,y$ of $\uc{D_2}$ such that $\uc{c}(gx) = \uc{c}(y)$. Now let $x_1 = \pi_{\uc{D_1}}(x)$ and $y_1 = \pi_{\uc{D_1}}(y)$. Since $c$ has a single vertex preimage, $\uc{c}$ only has one orbit of vertex preimages. Therefore, there exists $h$ a deck transformation such that $\uc{c}(y_1) = \uc{c}(hx_1)$.

First note that $h\in Stab(\uc{D_1}) = Stab(\uc{D_2})$ by the assumptions. Then, note that $\s(\uc{D_1}\mid \uc{D_2})\subseteq \s(hx_1\mid hx) = h\s(x_1\mid x) = h\s(\uc{D_1}\mid \uc{D_2})$ by Lemma~\ref{parallelseparator}. Since the first and last term have the same cardinality, the inclusion is an equality and $hx_1 = \pi_{\uc{D_1}}(hx)$. By Lemma~\ref{parallelseparator} again, $\s(hx_1\mid y_1) = \s(hx\mid y)$ contains only hyperplanes collapsed by $\uc{c}$. Therefore, $\uc{c}(hx) = \uc{c}(y)=\uc{c}(gx)$. Since $c$ is a homotopy equivalence, $g=h\in Stab(\uc{D_2})$. This proves that $D_2$ satisfies Assumption~$(4.)$ of Lemma~\ref{liftloop}. The converse holds symmetrically by exchanging $D_1$ and $D_2$.

Finally, assume that the assumptions of Lemma~\ref{liftloop} hold for both $D_1$ and $D_2$, and that $X'$ is isomorphic to $\SS$. Then, $c(D_1)$ and $c(D_2)$ correspond to locally convex subcomplexes of $\SS$, i.e.~embedded Salvetti complexes $\SS_{\Delta_1}$, $\SS_{\Delta_2}$ for some subgraphs $\Delta_1,\Delta_2$ of $\Gamma$. Since $\uc{D_1}$ and $\uc{D_2}$ are parallel, $\uc{c}(\uc{D_1})$ and $\uc{c}(\uc{D_2})$ are parallel in $\uc{X'}$. Since hyperplanes of $\uc{X'}$ have only one orbit of edges, this means that $c(D_1)$ and $c(D_2)$ have the same edges, i.e.~$\Delta_1$ and $\Delta_2$ have the same vertices. Hence $\Delta_1 = \Delta_2$ and $c(D_1) = c(D_2)$, proving the result.

\end{proof}

\section{Relative untwisted outer space}
\label{relative}

Let $\mathcal{G} = (A_{\Delta_1},\dots, A_{\Delta_m})$ be a collection of standard subgroups of $A_\Gamma$.

\begin{defi}
\label{relativeos}
We consider vertices of $K_\Gamma$ as spatial cube complexes with untwisted markings, modulo marking-preserving isomorphism, the equivalence class of a complex $X$ with marking $m$ being denoted $[X,m]\in K_\Gamma^{(0)}$. \emph{Marked Salvettis} are a particular type of vertices which are combinatorially isomorphic to the spatial cube complex $\SS$. Recall that the star $st(\sigma)$ of a marked Salvetti $\sigma$ is the simplicial subcomplex of $K_\Gamma$ spanned by $\sigma$ and its the adjacent vertices (i.e.~the marked spatial cube complex with a marking-preserving collapse to $\sigma$). Every marking mentioned from now on will implicitly be assumed untwisted.

Let $S^\mathcal{G}$ be the set of marked Salvettis $[S,m]$ in $K_\Gamma$ such that there exists a combinatorial isomorphism $\iota\colon S\to \SS$ making the outer automorphism $\iota_*\circ m^{-1}\colon A_\Gamma\to \pi_1\SS = A_\Gamma$ an element of $\unt(A_\Gamma;\,\mathcal{G})$. Note that this depends on $\iota$ in general (a different $\iota$ will still give an untwisted outer automorphism, but might not preserve $\mathcal{G}$). Let $K^\mathcal{G}\subseteq K_\Gamma$ be the union of stars of vertices of $S^\mathcal{G}$, i.e.~the subcomplex of $K_\Gamma$ spanned by marked spatial cube complexes $[X,m]$ such that there exists a homotopy equivalent collapse $c\colon X\to S$ and a combinatorial isomorphism $\iota \colon S\to \SS$ making the outer automorphism $\iota_*\circ c_*\circ m^{-1}\colon A_\Gamma\to \pi_1\SS = A_\Gamma$ an element of $\unt(A_\Gamma;\,\mathcal{G})$. Once again, this depends on $\iota$ and $c$ in general.
\end{defi}

\begin{rem}
Note that an untwisted marking on $\SS$ itself is an outer automorphism in $\unt(A_\Gamma)$. Every marked Salvetti in $K_\Gamma$ is of the form $[\SS,\varphi]$ for some $[\varphi]\in \unt(A_\Gamma)$. Every marked Salvetti in $S^\mathcal{G}$ is of the form $[\SS,\varphi]$ for some $[\varphi]\in \unt(A_\Gamma;\,\mathcal{G})$.

Clearly both $S^\mathcal{G}$ and $K^\mathcal{G}$ are preserved under the action of $\unt(A_\Gamma;\,\mathcal{G})$ and the action of $\unt(A_\Gamma;\,\mathcal{G})$ on $S^\mathcal{G}$ is transitive.

Finally, let $c\colon X\to S$ and $c'\colon X\to S'$ be two collapses of a spatial cube complex $X$, and let $m,M,m'$ be untwisted markings on $S,X,S'$ respectively such that $c$ and $c'$ preserve markings. Assume that $[S,m]\in S^\mathcal{G}$. Then $[S',m']\in S^\mathcal{G}$ if and only if there exist combinatorial isomorphisms $\iota\colon S\to \SS$ and $\iota'\colon S'\to \SS$ such that the outer automorphism $[\iota_* \circ c_*\circ (\iota'_*\circ c'_*)^{-1}]$ is in $\unt(A_\Gamma;\,\mathcal{G})$. The main goal of this section is to characterize combinatorially in $X$ when this can happen and derive some consequences. 
\end{rem}

In the next sections, we will prove a general result (Theorem~\ref{mccoolcyclic}) implying that $K^\mathcal{G}$ with the action of $\unt(A_\Gamma;\,\mathcal{G})$ is a \emph{spine of untwisted relative outer space}, meaning that the $K^\mathcal{G}$ is contractible and the action is proper and cocompact. We start by simply proving that $K^\mathcal{G}$ is connected.

\begin{defi}
    Say an edge of $K_\Gamma$ is \emph{elementary} when it corresponds to a marking-preserving collapse of a single hyperplane with domain a marked Salvetti. Recall that a \emph{Whitehead move} is an edge path of length $2$ in $K_\Gamma$ joining two marked Salvettis using two elementary edges, and a \emph{Whitehead path} is a concatenation of Whitehead moves going through each marked Salvetti at most once. Marked Salvettis one Whitehead move away from $[\SS, id]$ are exactly of the form $[\SS, W]$ for $W$ an outer Whitehead automorphism.
\end{defi}

\begin{lm}
\label{connected}
The complex $K^\mathcal{G}$ is connected. Any two marked Salvettis in $S^\mathcal{G}$ are joined by a Whitehead path in $K^\mathcal{G}$.
\end{lm}
\begin{proof}
Any vertex of $K^\mathcal{G}$ is at distance $1$ from a marked Salvetti via a marking-preserving collapse. Start from a marked Salvetti $[\SS,\varphi] \in S^\mathcal{G}$, with $[\varphi]\in \unt(A_\Gamma;\,\mathcal{G})$. If $[\varphi]$ is an outer Whitehead automorphism, $[\SS,\varphi]$ is one Whitehead move away from $[\SS,id]$, and the length $2$ path stays in $K^\mathcal{G}$.

In general, using Corollary~\ref{gen}, decompose $\varphi$ as a product of outer oriented graph permutations, outer folds and outer partial conjugations in $\unt(A_\Gamma;\,\mathcal{G})$. Each of the outer folds and outer partial conjugations decomposes as the product of an outer Whitehead automorphism and an outer inversion still in $\unt(A_\Gamma;\,\mathcal{G})$. Finally, as the set of Whitehead automorphisms is normalized by outer oriented graph permutations, write $\varphi = W_1\circ \dots\circ W_k\circ \psi$ where the $[W_i]\in \unt(A_\Gamma;\,\mathcal{G})$ are outer Whitehead automorphisms, and $\psi$ is an oriented graph permutation. As before, $[\SS,W_i]$ and $[\SS, id]$ are one Whitehead move away for all $i$, thus so are $[\SS, W_1\circ\dots\circ W_i]$ and $[\SS, W_1\circ\dots\circ W_{i-1}]$, which all belong to $S^\mathcal{G}$. This provides a Whitehead path in $K^\mathcal{G}$ from $[\SS, id]$ to $[\SS, W_1\circ\dots\circ W_k] = [\SS,\varphi]$, the last equality holding since $\psi$ stabilizes $[\SS,id]$.
\end{proof}

\begin{lm}
\label{whitehead}
Let $W$ be the Whitehead automorphism corresponding to the based Whitehead partition $(\mathbf{P},b)$. Let $c_\mathbf{P}\colon \SS^\mathbf{P}\to \SS$ be the collapse of the hyperplane labeled $\mathbf{P}$, and $c_b\colon \SS^\mathbf{P}\to \SS^\mathbf{P}_b$, the collapse of the hyperplane labeled $b$. The following are equivalent:
\begin{enumerate}
\item $[W]\in \unt(A_\Gamma;\,\mathcal{G})$
\item For each $A_{\Delta_i}\in \mathcal{G}$ such that $b\notin \Delta_i$, $single(\mathbf{P})\cap \Delta_i$ is empty, and at least one of $double(P)\cap \Delta_i$, $double(P^*)\cap \Delta_i$ is empty.
\item There exists $[\psi]$ an outer oriented graph automorphism such that $[W\circ \psi]\in \unt(A_\Gamma;\,\mathcal{G})$, i.e.~$[\SS,W]\in S^\mathcal{G}$.
\item $[\SS^\mathbf{P}_b,(c_{\mathbf{P}})_*\circ (c_b)_*^{-1}]\in S^\mathcal{G}$.
\end{enumerate}
\end{lm}
\begin{proof}
$(1.)\Rightarrow (2.)$ Let $[W]\in \unt(A_\Gamma;\,\mathcal{G})$. If for some $\Delta_i$, there exists $v$ in $single(\mathbf{P})\cap \Delta_i$, $W(v)\in \{bv,vb^{-1}, b^{-1}\}$ is conjugate to an element of $A_{\Delta_i}$ in which the total power in $b$ is $\pm 1$ ($b^{-1}$ appears only as $W(b)$). Hence $b$ appears in a cyclically reduced word representing that element, thus belongs to $\Delta_i$.

Likewise, if $v_1\in double(P)\cap \Delta_i$ and $v_2 \in double(P^*)\cap \Delta_i$, then $W(v_1v_2)=bv_1b^{-1}v_2$ is conjugate to an element of $A_{\Delta_i}$. As $link(b)=link(\mathbf{P})$ is disjoint from $double(P)$ and $double(P^*)$, $b$ does not commute with $v_1$ nor $v_2$. Hence the word $bv_1b^{-1}v_2$ is cyclically reduced and $b\in \Delta_i$ once again.

$(2.)\Rightarrow (1.)$ Fix $A_{\Delta_i}\in \mathcal{G}$. For every vertex $v$ of $\Delta_i$, $W(v)\in \{v, v^{-1}, vb^{-1}, bv, bvb^{-1}\}$. Thus if $b\in \Delta_i$, $W$ preserves $A_{\Delta_i}$. Otherwise, $single(\mathbf{P})\cap \Delta_i$ and one of $double(P)\cap \Delta_i$, $double(P^*)\cap \Delta_i$ are empty. If $double(P)\cap \Delta_i$ is empty, for every $v\in \Delta_i$, $W(v)=v$ and $W$ preserves $A_{\Delta_i}$ once again. Else, $double(P^*)\cap \Delta_i$ is empty and for every $v\in \Delta_i$, $W(v)=bvb^{-1}$ (as $v=bvb^{-1}$ when $v\in L$), thus a conjugate of $W$ preserves $A_{\Delta_i}$.

$(1.)\Rightarrow (3.)$ Pick $\psi = id$.

$(3.)\Rightarrow (1.)$ Pick an oriented graph automorphism $\psi$ representing $[\psi]$. For every standard generator $v\in V$, the standard generator $\psi(v)$ or $\psi(v)^{-1}$ appears in $[W\circ \psi(v)]$ with total power $\pm 1$. This implies that if $v\in \Delta_i$, $\psi(v)\in A_{\Delta_i}$. Hence $[\psi]\in \unt(A_\Gamma;\,\mathcal{G})$. Finally, $[W]\in \unt(A_\Gamma;\,\mathcal{G})$.

$(3.)\Leftrightarrow (4.)$ Since $b$ is a basepoint for $\mathbf{P}$, the hyperplanes labeled $b$ and $\mathbf{P}$ in $\SS^\mathbf{P}$ have the same set of transverse hyperplanes. By specialness of $\SS^\mathbf{P}$, each cube spanned by the edge $e_\mathbf{P}$ and other edges corresponds to a unique cube spanned by the edge $e_b$ and the same other edges, and vice-versa. Therefore, the exchange of $e_\mathbf{P}$ and $e_b$ extends to a combinatorial automorphism $\alpha$ of $\SS^\mathbf{P}$ fixing all the other edges and both vertices. There also exists a combinatorial isomorphism $\overline{\alpha}\colon \SS^\mathbf{P}_b\to \SS$ such that $\overline{\alpha}\circ c_b = c_\mathbf{P}\circ \alpha$. Then it is easy to check on standard generators that $[W] = [(c_\mathbf{P})_*\circ \alpha^{-1}_*\circ (c_\mathbf{P})_*^{-1}]=[(c_\mathbf{P})_*\circ (c_b)_*^{-1}\circ \overline{\alpha}^{-1}_*]$. By definition of vertices of $K_\Gamma$, $[\SS^\mathbf{P}_b, (c_\mathbf{P})_*\circ (c_b)_*^{-1}]=[\SS, (c_\mathbf{P})_*\circ (c_b)_*^{-1}\circ \overline{\alpha}_*^{-1}] = [\SS,W]$.
\end{proof}

The following proposition generalizes Lemma~\ref{whitehead} in the sense that it detects combinatorially when a marking change preserves the collection $\mathcal{G}$, for the wider class of marking changes induced by collapses of more than one hyperplane.

\begin{pro}
\label{caracwhitehead}
Let $X$ be a spatial cube complex and $c\colon X\to S$, $c'\colon X\to S'$ be two collapses that are homotopy equivalences with ranges isomorphic to $\SS$. Let $C, C'\subseteq X$ be the corresponding unique vertex preimages. Let $\iota \colon S\to \SS$ and $\iota'\colon S'\to \SS$ be combinatorial isomorphisms. The following are equivalent:
\begin{enumerate}
\item There exists $D_1,\dots , D_m\subseteq X$ a family of connected, locally convex subcomplexes such that for every $i\leq m$, the following hold:
    \begin{enumerate}
    \item[($a$)] $\iota\circ c(D_i) = \SS_{\Delta_i}$
    \item[($a'$)] $\iota'\circ c'(D_i) = \SS_{\Delta_i}$
    \item[($b$)] $D_i \cap C$ is convex in $C$
    \item[($b'$)] $D_i \cap C'$ is convex in $C'$
    \end{enumerate}
\item $[\iota_* \circ c_*\circ (\iota'_*\circ c'_*)^{-1}]\in \unt(A_\Gamma;\,\mathcal{G})$
\end{enumerate}
\end{pro}

Note that Properties~($b$) and ($b'$) can be reformulated thanks to Lemma~\ref{liftloop}.

\begin{proof}

$(1.)\Rightarrow(2.)$ Pick such $D_1,\dots,D_m$. For some fixed $i$, pick $\widetilde{*}\in D_i$. By Lemma~\ref{liftloop}, the map $\iota\circ c_{\mid D_i}$ induces an isomorphism between $\pi_1(D_i,\widetilde{*})$ and $\pi_1(\SS_{\Delta_i},*) = A_{\Delta_i}$. Likewise,
the map $\iota' \circ c'_{\mid D_i}$ induces an isomorphism between $\pi_1(D_i,\widetilde{*})$ and $\pi_1(\SS_{\Delta_i},*)$. Hence we found a choice of basepoint of $X$ so that $\iota_* \circ c_*\circ (\iota'_*\circ c'_*)^{-1}$ preserves $A_{\Delta_i}$. Since $i$ was arbitrary, $(2.)$ holds.

\bigskip

$(2.)\Rightarrow (1.)$ Fix lifts $f,f'\colon \uc{X}\to \uc{\SS}$ of $\iota\circ c$ and $\iota'\circ c'$ respectively between universal covers. Note that $f,f'$ are collapses between $\cat$ cube complexes. Let $i\leq m$ and let $A_i$, $A'_i$ be two arbitrary connected components of the preimage of $\SS_{\Delta_i}$ in the universal cover $\uc{\SS}$ (not necessarily distinct). Note that $A_i$ and $A'_i$ are convex, and that for every deck transformation $g$ of $\uc{\SS}$, if $gA_i$ and $A_i$ intersect, then $g$ is represented by the homotopy class of an edge cycle in $\SS_{\Delta_i}$ and $gA_i = A_i$. This means that $A_i$ is disjoint from its distinct translates. The same holds for $A'_i$.

Since $\iota \circ c$ and $\iota'\circ c'$ are homotopy equivalences, $f$ establishes a correspondence between the stabilizer of $f^{-1}(A_i)$ in the deck transformations of $\uc{X}$, and the stabilizer of $A_i$, a conjugate of $\pi_1\SS_{\Delta_i} = A_{\Delta_i}$ in $\pi_1\SS = A_\Gamma$. Likewise, $f'$ makes the stabilizer of $f'^{-1}(A'_i)$ correspond to a conjugate of $A_{\Delta_i}$ as well. By Assumption~$(2.)$, the stabilizers of $f^{-1}(A_i)$ and $f'^{-1}(A'_i)$ are conjugate (in the group of deck transformations of $\uc{X}$). Up choosing a different component $A'_i$, assume these stabilizers are the same subgroup $G$ of deck transformations. Note that $G$ is isomorphic to $A_{\Delta_i}$, hence $G$ is finitely generated. By Lemma~\ref{collapseupstairs}, both $f^{-1}(A_i)$ and $f'^{-1}(A'_i)$ are $G$-invariant convex subcomplexes. By Lemma~\ref{invariantminimal}, $f^{-1}(A_i)$ contains a minimal $G$-invariant convex (in $\uc{X}$) subcomplex $B_i$, $f'^{-1}(A'_i)$ contains a minimal $G$-invariant convex subcomplex $B'_i$, and $B_i$ and $B'_i$ are parallel in $\uc{X}$.

Let $g$ be a deck transformation of $\uc{X}$ and assume that $f(gB_i)$ and $f(B_i)$ intersect. Let $f_*g$ be the deck transformation of $\uc{\SS}$ corresponding to $g$ via $f$. Then, $f_*gA_i$ and $A_i$ intersect. By the argument above, $f_*gA_i=A_i$, thus $g\in G$ and $gB_i = B_i$. This fact has two consequences. First, $B_i$ is disjoint from its distinct translates, meaning that the quotient $B_i/G$ embeds in $X$. Second, let $D_i\simeq B_i/G$ be the image of this embedding. Since the restriction $B_i\to D_i$ of the universal covering $\uc{X}\to X$ is again a universal covering. Thus, $D_i$ satisfies Assumption~$(4.)$ of Lemma~\ref{liftloop} for the collapse $f$. Since $B_i$ is convex, $D_i$ is locally convex.

Symmetrically, $B'_i/G$ embeds in $X$ and the image $D'_i$ of this embedding is locally convex and satisfies Assumption~$(4.)$ of Lemma~\ref{liftloop} for the collapse $f'$. Moreover, $B_i$ and $B'_i$ are parallel, $X$ is special, and $\iota\circ c$, $\iota'\circ c'$ have range $\SS$. By Lemma~\ref{parallel}, both $D_i$ and $D'_i$ satisfy all the assumptions of Lemma~\ref{liftloop} for both collapses $\iota\circ c$ and $\iota'\circ c'$, and the equalities $\iota\circ c(D_i) = \iota\circ c(D'_i)$, $\iota'\circ c'(D_i) = \iota'\circ c'(D'_i)$ hold between locally convex subcomplexes of $\SS$. This proves in particular that $D_i$ satisfies ($b$) and ($b'$) from $(1.)$.

Finally, $f(B_i)$ is contained in $A_i$, and invariant under the stabilizer of $A_i$. Hence, $\iota\circ c(D_i)$ is locally convex, $\iota\circ c(D_i)\subset \SS_{\Delta}$, and the fundamental group of $\iota\circ c(D_i)$ contains $A_{\Delta_i}$. Therefore, for each vertex $v$ of $\Delta_i$, the subcomplex $\iota\circ c(D_i)$ contains the only edge of $\SS$ labelled $v$. By local convexity, $\iota\circ c(D_i)$ contains $\SS_{\Delta_i}$ entirely, proving that $D_i$ satisfies ($a$) from $(1.)$. By a symmetric argument, $\iota'\circ c'(D_i) = \iota'\circ c'(D'_i) = \SS_{\Delta_i}$, proving that $D_i$ satisfies ($a'$) from $(1.)$.

This construction provides for all $i$ a subcomplex $D_i$ satisfying all the requirements, and concludes the proof.
\end{proof}

The following two consequences of Proposition~\ref{caracwhitehead} will be of great use to prove contractibility of $K^\mathcal{G}$. Lemma~\ref{factorpath} is a factorization lemma in $K^\mathcal{G}$, and Lemma~\ref{prepeakreduction} is the main preliminary to peak reduction of Whitehead paths in $K^\mathcal{G}$.

\begin{lm}
\label{factorpath}
    Let $X$ be a spatial cube complex and $\H = \{H_1,\dots, H_k\}$, $\K$ two families of hyperplanes of $X$. Assume that the collapses $c_0\colon X\to S_0$ of $\H$ and $c_k\colon X\to S_k$ of $\K$ are both homotopy equivalences, and that there exist combinatorial isomorphisms $\iota_0 \colon S_0\to \SS$ and $\iota_k\colon S_k\to \SS$ such that $[(\iota_0)_* \circ (c_0)*\circ ((\iota_k)_*\circ (c_k)_*)^{-1}]\in \unt(A_\Gamma;\,\mathcal{G})$. 
    
    Then there exists an ordering $\K = \{K_1,\dots, K_k\}$ with the following properties:
    \begin{itemize}
        \item For all $0\leq j\leq k$, the collapse $c_j\colon X\to S_j$ of $\{H_1,\dots, H_{k-j},K_{k-j+1},\dots, K_k\}$ is a homotopy equivalence.
        \item There exist combinatorial isomorphisms $\iota_j\colon S_j\to \SS$ for $1\leq j<k$ such that for all $0\leq j<k$, $[(\iota_j)_* \circ (c_j)_*\circ ((\iota_{j+1})_*\circ (c_{j+1})_*)^{-1}]\in \unt(A_\Gamma;\,\mathcal{G})$.
    \end{itemize}

    Note that, in particular, for all $0\leq j\leq k$, $[(\iota_0)_* \circ (c_0)_*\circ ((\iota_{j})_*\circ (c_{j})_*)^{-1}]\in \unt(A_\Gamma;\,\mathcal{G})$.
\end{lm}

\begin{proof}
    The proof is by induction on $k$. Note that if $H_k\in \K$, setting $K_k=H_k$ and $\iota_1=\iota_0$ allows the induction to proceed in the range of the collapse of $K_k=H_k$. Assume now that $H_k\notin \K$. Let $C_0, C_k\subseteq X$ be the vertex preimages of $c_0,c_k$. Using Proposition~\ref{caracwhitehead} for $c_0$ and $c_k$, let $D_1,\dots, D_m\subseteq X$ as in the proposition. Let $d\colon X\to Y$ denote the collapse of $H_1,\dots H_{k-1}$, and $c\colon Y\to S_0$ the collapse of the image of $H_k$. By \cite{spatial}, Proposition~5.7 and Corollary~5.8, $Y$ is spatial and $c$ is a strong collapse. Moreover, the vertex preimage $C$ of $c$ contains only edges dual to $H_k$. Since $c$ is a strong collapse, this preimage is a single edge, thus $Y$ has exactly two vertices. The preimages of these two vertices in $X$ are convex subcomplexes of $C_0$. Thus, for all $i$, the intersection of $D_i$ with each of these two subcomplexes is convex in $C_0$. By Lemma~\ref{liftloop}, $D'_i = d(D_i)$ is connected and locally convex in $Y$ for all $i$.

    Clearly, if $D'_i$ contains both vertices of $Y$, $D_i$ contains vertices on either halfspace of $H_k$ in $C_0$. Since $D_i\cap C_0$ is convex in $C_0$, $D_i$ contains an edge dual to $H_k$, thus $D'_i\cap C = C$ is the edge $\overline{e}$ dual to $H_k$ in $Y$: the intersection is convex in $C$. Otherwise, $D'_i$ contains only one vertex of $Y$, and the intersection $D'_i\cap C$ is this vertex, which is convex in $C$. Moreover, in all cases, $\iota_0\circ c(D'_i) = \iota_0 \circ c_0(D_i) = \SS_{\Delta_i}$. This proves that $D'_i$, $c$ and $\iota_0$ satisfy Assumptions~$(a)$ and $(b)$ of Proposition~\ref{caracwhitehead}, $(1.)$.

    Since $H_k\notin \K$, edges of $X$ dual to $H_k$ are not in $C_k$. Choose such an edge $e$ and $p$ a geodesic path in $C_k$ with the same endpoints. The projection of $p$ in $Y$ joins the two vertices of $Y$, thus some edge of $p$ does not project to a loop in $Y$. Let $K_k\in \K$ be a hyperplane dual to such an edge. By Lemma~\ref{tubenew}, every hyperplane of $X$ transverse to $H_k$ is transverse to $K_k$. In particular, $H_k$ and $K_k$ are not transverse. Symmetrically, choose an edge $f$ in $X$ dual to $K_k$ and $q$ a geodesic path in $C_0$ with the same endpoints. As before, some edge of $q$ does not project to a loop in $Y$, thus has its endpoints in opposite halfspaces of $H_k$ in $C_0$. This edge of $q$ must then be dual to $H_k$. By Lemma~\ref{tubenew} again, every hyperplane of $X$ transverse to $K_k$ is transverse to $H_k$.
    
    By specialness of $Y$, there is a single edge $\overline{f}$ dual to $K_k$ in $Y$, and every cube spanned by the edge dual to $H_k$ and other edges corresponds to a unique cube spanned by the edge dual to $K_k$ and other edges and vice-versa. Therefore, there exists an involutive combinatorial automorphism $\alpha$ of $Y$ that exchanges $\overline{e}$ and $\overline{f}$ while fixing both vertices and all the other oriented edges. This proves that the collapse $c'\colon Y\to S_1$ of $K_k$ in $Y$ is a homotopy equivalence. Besides, $\alpha$ induces an isomorphism $\overline{\alpha}\colon S_0\to S_1$ in the sense that $c'\circ\alpha = \overline{\alpha} \circ c$. In particular, $\iota_1 = \iota_0\circ\overline{\alpha}^{-1} \colon S_1\to S_0\to \SS$ is a combinatorial isomorphism, and $\iota_1\circ c' = \iota_0\circ c\circ \alpha^{-1}$. Moreover, letting $c_1\colon X\to S_1$ be the collapse of $\{H_1,\dots, H_{k-1}, K_k\}$, $c_1$ is a composition $c'\circ d$ of homotopy equivalences, hence a homotopy equivalence itself.
    
    Recall that $\overline{e}$ and $\overline{f}$ are the only edges of $Y$ dual to $H_k$ and $K_k$ respectively. Assume $\overline{e}$ is in $D'_i$. Then some edge $e'$ parallel to $e$ is in $D_i$. Let $p'$ be a geodesic path in $C_k$ joining the endpoints of $e'$. By convexity of $D_i\cap C_k$, $p'$ is contained in $D_i$. By Corollary~\ref{tubemore}, every edge of $p$ is parallel to an edge of $p'$. In particular, $D_i$ contains some edge dual to $K_k$. In $Y$, its image $D'_i$ then contains $\overline{f}$, the only edge dual to $K_k$. Symmetrically, if $\overline{f}$ is in $D'_i$ then so is $\overline{e}$. Therefore, the $1$-skeleton of $D'_i$ is $\alpha$-invariant, and by local convexity of $D'_i$, the whole subcomplex $D'_i$ is $\alpha$-invariant. In particular, $\iota_1\circ c'(D'_i) = \iota_0\circ c(D'_i) = \SS_{\Delta_i}$. Moreover, if $D'_i$ contains both vertices of $Y$, $D'_i$ contains $\overline{e}$, hence contains $\overline{f}$ as well. Thus, the intersection of $D'_i$ with the vertex preimage $C' = \overline{f}$ of $c'$ is convex in $C'$. This proves that $D'_i$, $c'$ and $\iota_1$ satisfy Assumptions~$(a')$ and $(b')$ of Proposition~\ref{caracwhitehead}, $(1.)$.

    By Proposition~\ref{caracwhitehead}, $[(\iota_0)_* \circ (c_0)_*\circ ((\iota_{1})_*\circ (c_{1})_*)^{-1}] = [(\iota_0)_* \circ (c)_*\circ ((\iota_{1})_*\circ (c')_*)^{-1}] \in \unt(A_\Gamma;\,\mathcal{G})$. Thus $[(\iota_1)_* \circ (c_1)_*\circ ((\iota_{k})_*\circ (c_{k})_*)^{-1}]\in \unt(A_\Gamma;\,\mathcal{G})$ as well. Both collapses $c_1$, $c_k$ factor through the collapse of $\{K_k\}$ and its range $X'$. The induction now proceeds in $X'$.
\end{proof}

\begin{lm}
\label{prepeakreduction}
    Let $X$ be a spatial cube complex, and let $H_a,H_b,H_c,H_d$ be four distinct hyperplanes of $X$. Assume that both collapses $c_{ab}\colon X\to S_{ab}$ of $\{H_a,H_b\}$ and $c_{cd}\colon X\to S_{cd}$ of $\{H_c,H_d\}$ are homotopy equivalences with range isomorphic to $\SS$. Then, up to exchanging $H_a$ and $H_b$, both collapses $c_{ac}\colon X\to S_{ac}$ of $\{H_a,H_c\}$ and $c_{bd}\colon X\to S_{bd}$ of $\{H_b,H_d\}$ are homotopy equivalences with range isomorphic to $\SS$.

    Moreover, if there exists combinatorial isomorphisms $\iota_{ab}\colon S_{ab}\to \SS$ and $\iota_{cd}\colon S_{cd}\to \SS$ such that $[(\iota_{ab})_* \circ (c_{ab})_*\circ ((\iota_{cd})_*\circ (c_{cd})_*)^{-1}]\in \unt(A_\Gamma;\,\mathcal{G})$, then there exists combinatorial isomorphisms $\iota_{ac}\colon S_{ac}\to \SS$ and $\iota_{bd}\colon S_{bd}\to \SS$ such that $[(\iota_{ab})_* \circ (c_{ab})_*\circ ((\iota_{ac})_*\circ (c_{ac})_*)^{-1}]\in \unt(A_\Gamma;\,\mathcal{G})$ and $[(\iota_{ab})_* \circ (c_{ab})_*\circ ((\iota_{bd})_*\circ (c_{bd})_*)^{-1}]\in \unt(A_\Gamma;\,\mathcal{G})$.
\end{lm}
\begin{proof}
    Let $C_{ab},C_{cd}\subseteq X$ be the vertex preimages of $c_{ab}$ and $c_{cd}$. Since $c_{ab}$ is a strong collapse, the $\cat$ cube complex $C_{ab}$ has two hyperplanes ($H_a$ and $H_b$). Two cases arise depending on whether they are transverse or not:
    \begin{itemize}
        \item If $H_a$ and $H_b$ are transverse, $C_{ab}$ is a square, thus $X$ and $C_{cd}$ have four vertices. Since $C_{cd}$ has two hyperplanes as well ($H_c$ and $H_d$), $C_{cd}$ is a square, and $H_c$ and $H_d$ are transverse. Then some edge $e$ of $C_{ab}$ and some edge $f$ of $C_{cd}$ have the same endpoints. If $f$ is dual to $H_d$, up to exchanging $H_a$ and $H_b$, assume $e$ is dual to $H_a$ and $f$ to $H_d$ (if $f$ is dual to $H_c$, up to exchanging, assume $e$ is dual to $H_b$ and the argument is symmetrical). Applying Lemma~\ref{tubenew} twice, once for $e$, $c_{cd}$ and once for $f$, $c_{ab}$, yields that $H_a$ and $H_d$ have the same transverse hyperplanes in $X$. In particular, $H_a$ and $H_d$ are both transverse to $H_b$ and $H_c$. By specialness, an edge $g$ dual to $H_c$ cannot join opposite vertices of the square $C_{ab}$, otherwise it would span a square with both edges of $C_{ab}$ dual to $H_a$, making the hyperplane $H_a$ one-sided. Therefore, the edge $g$ has the same endpoints as an edge $h$ dual to $H_b$. By specialness again, the loops $ef^{-1}$ and $gh^{-1}$ span a product in $X$, which is an embedded combinatorial torus $T$. The exchange of edges $e$ and $f$ extends to a combinatorial automorphism of this torus fixing the four vertices, $g$, and $h$. Since $H_a$ and $H_d$ have the same transverse hyperplanes and $X$ is special, this automorphism extends further to an involutive automorphism $\alpha$ of $X$ fixing all the other edges. Since $\alpha$ exchanges $H_a$ and $H_d$ and preserves all the other hyperplanes, there exists combinatorial automorphisms $\overline{\alpha}\colon S_{ab}\to S_{bd}$ and $\overline{\alpha}'\colon S_{cd}\to S_{ac}$ such that $c_{bd} \circ \alpha = \overline{\alpha}\circ c_{ab}$ and $c_{ac} \circ \alpha = \overline{\alpha}'\circ c_{cd}$ . This proves that $c_{bd}$ and $c_{ac}$ are homotopy equivalences.

        \item If $H_a$ and $H_b$ are not transverse, $X$ has three vertices. Thus, $H_c$ and $H_d$ are not transverse, and both $C_{ab}$ and $C_{cd}$ consist of two incident edges. Therefore, some edge $e$ of $C_{ab}$ and some edge $f$ of $C_{cd}$ have the same endpoints. If $f$ is dual to $H_d$, up to exchanging $H_a$ and $H_b$, assume $e$ is dual to $H_a$ (if $f$ is dual to $H_c$, up to exchanging, assume $e$ is dual to $H_b$ and the argument is symmetrical). Applying Lemma~\ref{tubenew} twice as before, $H_a$ and $H_d$ have the same transverse hyperplanes in $X$. By specialness, the automorphism of the graph $C_{ab}\cup C_{cd}$ fixing all vertices, exchanging $e$ and $f$ and fixing all the other edges extends to an involutive automorphism $\alpha$ of $X$ fixing all the other edges. Moreover, there exists combinatorial automorphisms $\overline{\alpha}\colon S_{ab}\to S_{bd}$ and $\overline{\alpha}'\colon S_{cd}\to S_{ac}$ as before. Thus $c_{bd}$ and $c_{ac}$ are homotopy equivalences exactly as in the first case.
    \end{itemize}
    
    Assume now the existence of $\iota_{ab}$, $\iota_{cd}$ as in the statement and let $\iota_{ac} = \iota_{cd}\circ \overline{\alpha}'^{-1} \colon S_{ac}\to S_{cd}\to \SS$ and $\iota_{bd} =  \iota_{ab} \circ \overline{\alpha}^{-1}\colon S_{bd}\to S_{ab}\to \SS$, so that $\iota_{ac}\circ c_{ac} = \iota_{cd}\circ c_{cd}\circ \alpha$ and $\iota_{bd}\circ c_{bd} = \iota_{ab}\circ c_{ab}\circ \alpha$. Use Proposition~\ref{caracwhitehead} to find connected locally convex $D_i\subseteq X$ for all $i\leq m$ such that $D_i\cap C_{ab}$ and $D_i\cap C_{cd}$ are convex in $C_{ab}$ and $C_{cd}$ respectively, and $\iota_{ab}\circ c_{ab}(D_i) = \iota_{cd}\circ c_{cd}(D_i) = \SS_{\Delta_i}$. For a fixed $i$, by convexity of $D_i\cap C_{ab}$ and $D_i\cap C_{cd}$, $D_i$ contains the edge $e$ if and only if it contains both its endpoints, if and only if it contains the edge $f$. By local convexity, $D_i$ contains a cube spanned by $e$ and other edges if and only if $D_i$ contains a cube spanned by $f$ and the same edges. In the case where $H_a$ and $H_b$ are transverse, the same is true for the other edge $e'$ parallel to $e$ in $C_{ab}$ and the other edge $f'$ parallel to $f$ in $C_{cd}$. Since any cube not containing $e$, $f$, $e'$ or $f'$ is fixed by $\alpha$, $D_i$ is $\alpha$-invariant. Therefore, $C_{ac}\cap D_i = \alpha(C_{cd})\cap D_i = \alpha(C_{cd}\cap D_i)$ is convex in $\alpha(C_{cd}) = C_{ac}$ and $C_{bd}\cap D_i = \alpha(C_{ab})\cap D_i)$ is convex in $C_{bd}$. Moreover, $\iota_{ac}\circ c_{ac}(D_i) = \iota_{cd}\circ c_{cd}\circ \alpha(D_i) = \SS_{\Delta_i}$ and likewise $\iota_{bd}\circ c_{bd}(D_i) = \iota_{ab}\circ c_{ab}\circ \alpha(D_i) =\SS_{\Delta_i}$. Thus, by Proposition~\ref{caracwhitehead}, $[(\iota_{ab})_* \circ (c_{ab})_*\circ ((\iota_{ac})_*\circ (c_{ac})_*)^{-1}]\in \unt(A_\Gamma;\,\mathcal{G})$ and $[(\iota_{ab})_* \circ (c_{ab})_*\circ ((\iota_{bd})_*\circ (c_{bd})_*)^{-1}]\in \unt(A_\Gamma;\,\mathcal{G})$ as required.
\end{proof}

\section{Adding fixed cyclic subgroups}
\label{fixedZ}

Let $\mathcal{G} = (A_{\Delta_1},\dots A_{\Delta_p})$ be a collection of standard subgroups of $A_\Gamma$ and let $\mathcal{H} = (\gen{h_1},\dots,\gen{h_q})$ be a finite collection of cyclic subgroups of $A_\Gamma$. Our goal is to construct the following object:

\begin{thm}
\label{mccoolcyclic}
There exists a contractible subcomplex $K^{\mathcal{G}}_{\mathcal{H}^t}$ of the spine of untwisted outer space $K_\Gamma$ on which $\unt(A_\Gamma; \mathcal{G},\mathcal{H}^t)$ acts properly and cocompactly. In particular, this group is of type VF.
\end{thm}

\begin{ex}
\label{conjugatingautomorphisms}
An automorphism $\varphi\in \aut(A_\Gamma)$ is \emph{conjugating} when $\varphi(v)$ is conjugate to $v$ for all $v\in V$. By \cite{laurence}, Theorem~2.2, the subgroup of conjugating automorphisms is exactly the subgroup generated by partial conjugations in $A_\Gamma$. In particular, its image in $\out(A_\Gamma)$ is contained in $\unt(A_\Gamma)$. Therefore, this image is the McCool group $\unt(A_\Gamma;\,\{\gen{v}\mid v\in V\}^t)$, hence is of type VF by Theorem~\ref{mccoolcyclic}.
\end{ex}

Together with Lemma~\ref{autmccool}, we get directly a \emph{spine of untwisted auter space} for $A_\Gamma$ (as a subcomplex of the spine of untwisted outer space for $A_\Gamma*\ZZ$):

\begin{cor}
\label{auterspace}
There exists a contractible simplicial complex $L_\Gamma$ on which $\uaut(A_\Gamma)$ acts properly and cocompactly. In particular, this group is of type VF.
\end{cor}

Note that type VF was already an algebraic consequence of the fact that $\unt(A_\Gamma)$ is of type VF and $\inn(A_\Gamma)$ is of type F, as a right-angled Artin group (\cite{geoghegan}, Theorem~7.3.4).

\begin{defi}
For $X$ a spatial cube complex, $[g]\in \pi_1X$ a conjugacy class, let $\ell_X(g)$ be the smallest length of an edge cycle in the free homotopy class representing $[g]$, or equivalently the translation length (for the combinatorial metric) of any representative $g$ acting on $\uc{X}$. In particular, for $g\in A_\Gamma$, let $\ell(g)=\ell_{\SS}(g)$, which is the minimal length of any cyclic word in the standard generators representing the conjugacy class of $g$. Complete $(h_1,\dots, h_q)$ into an infinite sequence $(h_i)_{i\geq 1}$ of elements of $A_\Gamma$ containing at least one representative per conjugacy class.

Given a marked Salvetti $\sigma = [X,m]\in K_\Gamma$, its \emph{lexicographic norm} will be the following sequence of non-negative integers: \[N(\sigma) = (\ell_X(m^{-1}(h_i)))_{i\geq 1}\]
If $[X,m]=[X',m']$ is a marked Salvetti, there exists a combinatorial isomorphism $\alpha\colon X\to X'$ such that $[\alpha_*]=[m'^{-1}\circ m]$. Since $\alpha$ preserves lengths, for all $h\in A_\Gamma$, $\ell_X(m^{-1}(h)) = \ell_{X'}(\alpha_*\circ m^{-1}(h)) = \ell_{X'}(m'^{-1}(h))$, hence $N$ is well-defined.
\end{defi}

\begin{rem}
    This norm is related to the norm $\n{\cdot}$ defined in \cite{untwistedos}, Definition~6.1 by a shuffling and possible duplicating of entries. We will show that most of the properties proved in \cite{untwistedos} for $\n{\cdot}$ also hold for $N$.
\end{rem}

\begin{lm}
\label{wellorder}
The map $N$ is injective on the set of marked Salvettis in $K_\Gamma$. Setting $\sigma \leq \sigma'$ if and only if $N(\sigma)\leq N(\sigma')$ in the lexicographic order defines a well-ordering of the set of marked Salvettis in $K_\Gamma$
\end{lm}

\begin{proof}
Let $\sigma =[\SS,\varphi],\sigma'=[\SS,\varphi']$ be two marked Salvettis and assume $N(\sigma) = N(\sigma')$. Then for every $i\geq 1$, $\ell(\varphi^{-1}\circ \varphi'(h_i)) = \ell(h_i)$. In particular for every $h\in A_\Gamma$ such that $\ell(h)\leq 2$, $\ell(\varphi^{-1}\circ \varphi'(h)) = \ell(h)$. By Lemma~6.2 in \cite{untwistedos}, $\varphi^{-1}\circ \varphi'$ is an outer oriented graph permutation, thus $\sigma = \sigma'$, proving that $N$ is injective.

Suppose the existence of a decreasing sequence $(\sigma_n = [\SS,\varphi_n])_n$, and choose arbitrary representatives $\varphi_n\in \aut(A_\Gamma)$. The non-increasing positive integer sequence $(\ell(\varphi_n^{-1}(h_1)))_n$ is eventually constant. Thus the sequence $(\ell(\varphi_n^{-1}(h_2)))_n$ is eventually non-increasing, and then itself eventually constant. Inductively, for every $i$, the sequence $(\ell(\varphi_n^{-1}(h_i)))_n$ is eventually constant. This implies that the following quantity: \[\max_i \ell(\varphi_n^{-1}(v_i)) + \dfrac{\dim(\SS)}{2}\max_{i<j} \ell(\varphi_n^{-1}(v_iv_j))+2\dim(\SS)\] is bounded by some constant $C$ independent of $n$.

Apply Proposition~\ref{bounded_displacement} for all $n$ to $A_\Gamma$ acting on $\USS$, with $a_i = \varphi_n^{-1}(v_i)$ the preimages of standard generators. There exists $(x_n)_n\in \USS$ a sequence of vertices such that $d(x_n,\varphi_n^{-1}(v_i)x_n)\leq C$ for all $i,n$. Since $A_\Gamma$ acts transitively on vertices of $\USS$, up to composing $\varphi_n$ by an inner automorphism, assume $(x_n)_n$ is constant equal to $x_0$. Since $A_\Gamma$ acts freely on $\USS$, only finitely many group elements $g$ satisfy $d(x_0,gx_0)\leq C$. Thus there are only finitely many choices for all the $\varphi_n^{-1}(v_i)$. This implies the existence of $n\neq n'$ with $\varphi_n =\varphi_{n'}$, hence $\sigma_n =\sigma_{n'}$ and a contradiction.
\end{proof}

Beware that this order type may be more complicated that in \cite{untwistedos}, where the set of marked Salvettis has simply the order type of the integers (\cite{untwistedos}, Corollary~6.21). We will however still be able to proceed using transfinite induction. The following construction of the spine of untwisted outer space relative to $\mathcal{G}$ and $\mathcal{H}$ is close to the construction in \cite{mccoolwhitehead} of a relative outer space for free groups.

\begin{defi}
For every marked Salvetti $\sigma\in K_\Gamma$, let $K^\mathcal{G}_{<\sigma}$ be the union of all stars (in $K_\Gamma$) of marked Salvettis in $S^\mathcal{G}$ that are (strictly) smaller than $\sigma$. It is clear that $K^\mathcal{G}_{<\sigma}$ is a subcomplex of $K^\mathcal{G}$. Let $\sigma_0 = [\SS,\varphi_0]$ be the minimum of $S^\mathcal{G}$ for our well-ordering, so that $K^\mathcal{G}_{<\sigma}=\emptyset \Leftrightarrow \sigma = \sigma_0$.

Let $S^\mathcal{G}_{\mathcal{H}^t}$ be the set of marked Salvettis in $S^\mathcal{G}$ where each conjugacy class in $\mathcal{H}$ has the same length as in $\sigma_0$. It is a lower subset of $S^\mathcal{G}$ for the well-ordering:
\[S^\mathcal{G}_{\mathcal{H}^t} = \{[X,m]\in S^\mathcal{G} \mid \forall i\leq q,\, \ell_X(m^{-1}(h_i)) = \ell(\varphi_0^{-1}(h_i))\}\]
Let $K^\mathcal{G}_{\mathcal{H}^t}$ be the union of stars (in $K_\Gamma$) of all vertices in $S^\mathcal{G}_{\mathcal{H}^t}$. If $S^\mathcal{G}_{\mathcal{H}^t}\neq S^\mathcal{G}$, let $\sigma_1$ be the minimum of the complement $S^\mathcal{G}\setminus S^\mathcal{G}_{\mathcal{H}^t}$ so that $K^\mathcal{G}_{\mathcal{H}^t} = K^\mathcal{G}_{<\sigma_1}$. Otherwise, $K^\mathcal{G}_{\mathcal{H}^t} = K^\mathcal{G}$.
\end{defi}

\begin{lm}
The complex $K^\mathcal{G}_{\mathcal{H}^t}$ is $\unt(A_\Gamma;\,\mathcal{G},\mathcal{H}^t)$-stable with only finitely many orbits of vertices.
\end{lm}
\begin{proof}
Note first that if $\mathcal{H}$ is empty, we know from Definition~\ref{relativeos} that $K^\mathcal{G}$ is $\unt(A_\Gamma;\,\mathcal{G})$-stable and has a single orbit of marked Salvettis.

If $\mathcal{H}$ is non-empty, $S^\mathcal{G}_{\mathcal{H}^t}$ is $\unt(A_\Gamma;\,\mathcal{G},\mathcal{H}^t)$-stable, and so is $K^\mathcal{G}_{\mathcal{H}^t}$. Let $C$ be the finite set of conjugacy classes in $A_\Gamma$ of length at most $\displaystyle \max_{1\leq i\leq q} \ell(\varphi_0^{-1}(h_i))$, let $D = \{[\varphi]\in \unt(A_\Gamma;\,\mathcal{G})\mid  \forall i\leq q,\, \ell(\varphi^{-1}(h_i)) = \ell(\varphi_0^{-1}(h_i))\}$ and consider the following map:
\[\begin{aligned}
D&\to C^q\\
[\varphi]&\mapsto ([\varphi^{-1}(h_i)])_{1\leq i\leq q}
\end{aligned}
\]

Two outer automorphisms $[\varphi],[\psi]\in D$ have the same image under this map if and only if $[\varphi\circ \psi^{-1}(h_i)] = [h_i]$ for every $i\leq q$, or equivalently $[\varphi\circ\psi^{-1}]\in \unt(A_\Gamma;\,\mathcal{G},\mathcal{H}^t)$. In that case, the marked Salvettis $[\SS,\varphi],[\SS,\psi]\in S^\mathcal{G}_{\mathcal{H}^t}$ have the same $\unt(A_\Gamma;\,\mathcal{G},\mathcal{H}^t)$-orbit. Therefore, $S^\mathcal{G}_{\mathcal{H}^t}$ contains at most $\card{C^q}$ distinct such orbits (by picking any marking representatives in the equivalence classes). Hence $K^\mathcal{G}_{\mathcal{H}^t}$ has finitely many $\unt(A_\Gamma;\,\mathcal{G},\mathcal{H}^t)$-orbits of marked Salvettis.

Finally, every vertex in $K^\mathcal{G}_{\mathcal{H}^t}$ is adjacent to a marked Salvetti in $S^\mathcal{G}_{\mathcal{H}^t}$, which has no more neighbors than the finite number of sets of pairwise compatible Whitehead partitions of $V^\pm$, proving the result.
\end{proof}

\section{Proof of contractibility}
\label{contractibility}

We found a subcomplex $K^\mathcal{G}_{\mathcal{H}^t}\subseteq K_\Gamma$ on which $\unt(A_\Gamma;\,\mathcal{G},\mathcal{H}^t)$ acts properly and cocompactly, and the latter is virtually torsion-free as a subgroup of $\out(A_\Gamma)$. To prove Theorem~\ref{mccoolcyclic}, it remains only to see that $K^\mathcal{G}_{\mathcal{H}^t}$ is contractible. This generalizes the main result of \cite{untwistedos} (Theorem~6.24), and our proof will be very parallel to \cite{untwistedos}, Section~6. The proofs there use the terminology of blow-ups and Whitehead partitions, and we will mirror this use when reemploying them.

\begin{defi}
    An elementary edge $e$ of $K_\Gamma$ incident to a marked Salvetti $\sigma$ is \emph{$N\mathcal{G}$-reductive} if there exists a Whitehead move in $K^\mathcal{G}$ containing $e$, with endpoints $\sigma$ and $\sigma'$, such that $N(\sigma')<N(\sigma)$. Such a Whitehead move is an \emph{$N\mathcal{G}$-reductive move} for $\sigma$.

    Two (possibly equal) elementary edges of $K_\Gamma$ incident to the same marked Salvetti $\sigma$, with other endpoints $\alpha,\beta$ are \emph{compatible} if there exists a vertex $\gamma$ of $K_\Gamma$ that collapses onto both $\alpha$ and $\beta$. They are \emph{incompatible} otherwise. Let $e,e'$ be two edges of $K_\Gamma$ (not necessarily elementary). Say that $e'$ \emph{factors} $e$ when $e'$ is elementary and the edges $e,e'$ are part of a triangle in $K_\Gamma$. In that case, some collapse map representing $e$ factors through some collapse map representing $e'$.
    \medskip

    Let $X$ be a spatial cube complex, $H$ a hyperplane of $X$ and $[g]\in \pi_1X$ a conjugacy class. Define $\ell_{X, H}(g)$ as the number of edges dual to $H$ in a shortest edge loop representing $[g]$. By Lemma~\ref{axes}, this number is independent from the chosen representative loop. Given a marking $m$ on $X$, define the following sequence of non-negative integers:
    \[N_H(X,m) = (\ell_{X,H}(m^{-1}(h_i)))_{i\geq 1}\]
    
    Let $e$ be an elementary edge, and choose a marking-preserving collapse $(X,m)\to (S,m')$ representing $e$. Let $H$ be the collapsed hyperplane of $X$. Define $N^+(e) = N_H(X,m)$, which is independent from the chosen representative collapse for $e$.

    Finally, for $\varphi\in\unt(A_\Gamma;\,\mathcal{G})$ say a based Whitehead partition $(\mathbf{P},b)$ is \emph{$\varphi$-$N\mathcal{G}$-reductive} if it induces an $N\mathcal{G}$-reductive Whitehead move at $[\SS,\varphi]$. In other words, if $N([\SS^\mathbf{P}_b,\varphi'])<N([\SS,\varphi])$ where $c_\mathbf{P}\colon \SS^\mathbf{P}\to \SS$ and $c_b\colon \SS^\mathbf{P}\to \SS^\mathbf{P}_b$ are marking-preserving collapses of the hyperplanes labelled $\mathbf{P}$ and $b$ respectively, for some markings $M$ on $\SS^\mathbf{P}$ and $\varphi'$ on $\SS^\mathbf{P}_b$ and moreover $[\SS^\mathbf{P}_b,\varphi']\in S^\mathcal{G}$.
\end{defi}

\begin{rem}
\label{collapseeverywhere}
    Let $c\colon (X,m) \to (X',m')$ be a marking-preserving collapse, and let $H$ be a hyperplane of $X$ not collapsed by $c$. By Corollary~\ref{collapsedownstairs}, for all $i$, a shortest loop representing $[m^{-1}(h_i)]$ is mapped by $c$ to a shortest loop representing $[m'^{-1}(h_i)]$. Therefore, $N_H(X,m) = N_H(X',m')$. In particular, $N_H(X,m) = N^+(e)$ for any elementary edge $e$ induced by a collapse of $H$ factoring some collapse of $X$.

    Consider a single blow-up $\SS^\mathbf{P}$ with a marking $M$ and two homotopy equivalent collapses $c_\mathbf{P}\colon \SS^\mathbf{P}\to \SS$, $c_v\colon \SS^\mathbf{P}\to \SS^\mathbf{P}_v$ of the hyperplanes labeled $\mathbf{P}$ and $v\in V$ respectively. We will still denote these hyperplanes by $\mathbf{P}$ and $v$ for simplicity of notation. Let $\varphi\in \unt(A_\Gamma)$ be a marking on $\SS$ making $c_\mathbf{P}$ marking-preserving. Let $e$ be the edge of $K_\Gamma$ represented by $c_\mathbf{P}$ and $f$ the edge represented by $c_v$. By definition and the previous remark, $N^+(e) = N_{\mathbf{P}}(\SS^\mathbf{P},M_\mathbf{P})$ and $N^+(f) = N_v(\SS^\mathbf{P},M_\mathbf{P}) = N_v(\SS,\varphi)$. In the notation of \cite{untwistedos}, Section~6.2, if $i\geq 1$ and $w_i$ a minimal length word representing $[\varphi^{-1}(h_i)]$, the $i$-th coordinate of $N^+(e)$ is $|\mathbf{P}|_w$ and the $i$-th coordinate of $N^+(f)$ is $|v|_w$. Thus $|\mathbf{P}|_\sigma$ and $|\mathbf{v}|_\sigma$ correspond to $N^+(e)$ and $N^+(f)$ up to the same shuffling and possible duplicating of entries that relates $N(\sigma)$ with $\n{\sigma}$.
\end{rem}

\begin{lm}[see \cite{untwistedos}, Corollary~6.6]
\label{changenorm}
    Let $\sigma,\sigma'$ be two marked Salvettis in $K_\Gamma$ and $\alpha$ a vertex of $K_\Gamma$ with an edge $e$ to $\sigma$ and an edge $f$ to $\sigma'$. The following equality holds:
    \[N(\sigma') = N(\sigma) + \sum_{e'\text{ factoring }e} N^+(e') - \sum_{f'\text{ factoring }f} N^+(f')\]
\end{lm}
\begin{proof}
    Let $X$ be a spatial cube complex and $M$ a marking on $X$ with $[X,M]=\alpha$. Let $c\colon X\to S$ and $c'\colon X\to S'$ be homotopy equivalent collapses representing the edges $e$ and $f$. Let $m, m'$ be markings on $S,S'$ making $c$ and $c'$ marking-preserving. Let $\H$ be the set of hyperplanes of $X$ collapsed by $c$ and $\K$ the set of hyperplanes of $X$ collapsed by $c'$.

    Let $i\geq 1$. Let $\gamma$ be a minimal length edge loop in $X$ whose free homotopy class represents the conjugacy class $[M^{-1}(h_i)]$. The length of $\gamma$ is $\ell_X(M^{-1}(h_i)) = \displaystyle\sum_{H\text{ hyperplane of }X} \ell_{X,H}(M^{-1}(h_i))$. The sum $\displaystyle\sum_{H\notin \H} \ell_{X,H}(M^{-1}(h_i))$ corresponds to the length of $c\circ \gamma$ in $S$. By Corollary~\ref{collapsedownstairs}, $c\circ \gamma$ is of minimal length representing the conjugacy class $[m^{-1}(h_i)]$. Thus $\displaystyle\sum_{H\notin \H} \ell_{X,H}(M^{-1}(h_i)) = \ell_{S_1}(m^{-1}(h_i))$, which is the $i$-th coordinate of $N(\sigma)$.
    
    Now, edges factoring $e$ are in one-to-one correspondance with hyperplanes of $\H$. Let $e'$ be an edge factoring $e$ and $H\in \H$ the corresponding hyperplane. By Remark~\ref{collapseeverywhere}, $\ell_{X,H}(M^{-1}(h_i))$ is the $i$-th coordinate of $N^+(e')$.

    The previous argument holds symmetrically for $c'\colon X\to S'$, $m'$, and $\K$. Thus $\ell_X(M^{-1}(h_i))$ is both the sum of the $i$-th coordinate of $N(\sigma)$ with all the $i$-th coordinates of $N^+(e')$, $e'$ factoring $e$ and the sum of the $i$-th coordinate of $N(\sigma')$ with all the $i$-th coordinates of $N^+(f')$, $f'$ factoring $f$. This proves the equality coordinatewise.
\end{proof}

\begin{lm}
\label{peak}
    Let $\sigma,\tau_1,\tau_2$ be distinct marked Salvettis in $S^\mathcal{G}$. Assume there exists an $N\mathcal{G}$-reductive Whitehead move from $\sigma$ to $\tau_1$ with first edge $e_1$, and an $N\mathcal{G}$-reductive Whitehead move from $\sigma$ to $\tau_2$ with first edge $e_2$. Assume $e_1$ and $e_2$ are compatible. Then there exists a Whitehead path of one or two Whitehead moves, joining $\tau_1$ to $\tau_2$, whose possible middle marked Salvetti is in $S^\mathcal{G}$ and has smaller norm than $\sigma$.
\end{lm}
\begin{proof}
    By compatibility, let $\alpha$ be a vertex of $K^\mathcal{G}$ adjacent to $\sigma$ via an edge $f$ such that $e_1$ and $e_2$ factor $f$. In particular, $\alpha$ is greater (in the poset of vertices of $K^\mathcal{G}$) than all vertices of the two considered Whitehead moves. Thus, $\alpha$ is adjacent to $\tau_1$ and $\tau_2$.

    Let $X$ be a spatial cube complex with a marking $M$ such that $[X,M] = \alpha$. Let $c_1\colon X\to S_1$, $c_2\colon X\to S_2$, $c_3\colon X\to S_3$ be homotopy equivalent collapses representing the edges from $\alpha$ to $\tau_1$, $\tau_2$ and $\sigma$ respectively, and let $m_1$, $m_2$, $m_3$ be corresponding markings on $S_1$, $S_2$, $S_3$ making the collapses marking-preserving, i.e $m_1\circ (c_1)_* = m_2\circ (c_2)_* = m_3\circ (c_3)_* = M$. Since $[S_1,m_1]=\tau_1$ and $[S_2,m_2]=\tau_2$ are in $S^\mathcal{G}$, there exists by definition two combinatorial isomorphisms $\iota_1\colon S_1\to \SS$ and $\iota_2\colon S_2\to \SS$ with $[(\iota_1)_*\circ m_1^{-1}], [(\iota_2)_*\circ m_2^{-1}]\in \unt(A_\Gamma;\,\mathcal{G})$. In particular, $[(\iota_1)_*\circ (c_1)_*\circ ((\iota_2)_*\circ (c_2)_*)^{-1}] = [(\iota_1)_*\circ m_1^{-1}\circ ((\iota_2)_*\circ m_2^{-1})^{-1}]\in\unt(A_\Gamma;\,\mathcal{G})$.

    Let $H_a,H_b$ be the two hyperplanes of $X$ collapsed by $c_1$ and $H_c,H_d$ the two hyperplanes of $X$ collapsed by $c_2$. Assume that $H_b$ and $H_c$ are collapsed by $c_3$, the other cases being symmetrical. If $H_a=H_d$, then the collapse of this hyperplane defines the middle vertex of a Whitehead move joining $\tau_1$ and $\tau_2$, proving the claim. Assume henceforth that $H_a$ and $H_d$ are distinct. Moreover $H_a$ and $H_c$ are distinct, otherwise $\tau_1 = \sigma$. Likewise, $H_b$ and $H_d$ are distinct, otherwise $\tau_2 = \sigma$. Finally, since the three families $\{H_a,H_b\}$, $\{H_c,H_d\}$ and $\{H_b,H_c\}$ are collapsible, $H_a,H_b,H_c,H_d$ are all pairwise distinct.

    By Lemma~\ref{prepeakreduction}, up to exchanging $H_a$ and $H_b$, the collapse $c_{ac}$ of $\{H_a,H_c\}$ and the collapse $c_{bd}$ of $\{H_b,H_d\}$ in $X$ yield complexes $T$ and $T'$ respectively, isomorphic to $\SS$, with markings $n$, $n'$ making the collapses marking-preserving, and isomorphisms $\iota\colon T\to \SS$, $\iota'\colon T'\to \SS)$ with the following properties:
    \[[(\iota_1)_*\circ (c_1)_*\circ (\iota_*\circ (c_{ac})_*)^{-1}]\in\unt(A_\Gamma;\,\mathcal{G})\]
    \[[(\iota_1)_*\circ (c_1)_*\circ (\iota'_*\circ (c_{bd})_*)^{-1}]\in\unt(A_\Gamma;\,\mathcal{G})\]
    
    In particular, since $[(\iota_1)_*\circ m_1^{-1}]\in \unt(A_\Gamma;\,\mathcal{G})$, both $[\iota_*\circ n^{-1}]$ and $[\iota'_*\circ n'^{-1}]$ are in $\unt(A_\Gamma;\,\mathcal{G})$ as well. Thus, $[T,n], [T',n']\in S^\mathcal{G}$. Moreover, since $\tau_1,\tau_2, [T,n], [T',n']$ are obtained from $[X,M]$ by collapsing distinct families of hyperplanes, these four points of $S^\mathcal{G}$ are distinct by \cite{spatial}, Lemma~5.11.
    
    By Remark~\ref{collapseeverywhere} and Lemma~\ref{changenorm}, the following equalities hold:
    \[\begin{aligned}
        N([T,n]) &= N(\tau_1) + N_{H_b}(X,M)-N_{H_c}(X,M)\\
        N([T',n']) &= N(\tau_2) - N_{H_b}(X,M)+N_{H_c}(X,M)
    \end{aligned}\]
    The two sequences $N_{H_b}(X,M)-N_{H_c}(X,M)$, $N_{H_c}(X,M)-N_{H_b}(X,M)$ are opposite and not null, by injectivity of $N$. Thus, one of the sequences is smaller than the null sequence in the lexicographic order. Therefore, one of $[T,n]$, $[T',n']$ has smaller norm than one of $\tau_1$, $\tau_2$. In particular, it has smaller norm than $\sigma$. This finishes the proof since both $[T,n]$ and $[T',n']$ admit Whitehead moves to $\tau_1$ and $\tau_2$.
\end{proof}

\begin{lm}[see \cite{untwistedos}, Corollary~6.8, Case 1]
\label{factoring}
    Let $\sigma,\tau$ be two marked Salvettis in $S^\mathcal{G}$ and $\alpha$ a vertex of $K_\Gamma$ with an edge $e$ to $\sigma$ and an edge $f$ to $\tau$. Assume $N(\tau)<N(\sigma)$. Then there exists an $N\mathcal{G}$-reductive elementary edge incident to $\sigma$ factoring $e$.
\end{lm}
\begin{proof}
    Let $X$ be a spatial cube complex and $M$ a marking on $X$ with $[X,M]=\alpha$. Let $c\colon X\to S$ and $c'\colon X\to S'$ be homotopy equivalent collapses representing the edges $e$ and $f$. Let $m, m'$ be markings on $S,S'$ making $c$ and $c'$ marking-preserving, i.e.~$m\circ c_* = m'\circ c'_* = M$. Let $\H$ be the set of hyperplanes of $X$ collapsed by $c$ and $\K$ the set of hyperplanes of $X$ collapsed by $c'$. Pick $H_1,\dots, H_k$ an arbitrary ordering of $\H$. By definition of $S^\mathcal{G}$, there exist combinatorial isomorphisms $\iota\colon S\to \SS$ and $\iota'\colon S'\to \SS$ with $[\iota_*\circ m^{-1}], [\iota'_*\circ m'^{-1}]\in \unt(A_\Gamma;\,\mathcal{G})$. In particular, $[\iota'_*\circ c'_*\circ (\iota_*\circ c_*)^{-1}] = [\iota'_*\circ m'^{-1}\circ (\iota_*\circ m^{-1})^{-1}] \in\unt(A_\Gamma;\,\mathcal{G})$.

    By Lemma~\ref{factorpath}, there exists an ordering $K_1,\dots,K_k$ of elements of $\K$ yielding a sequence of homotopy equivalent collapses $c_i\colon X\to S_i$, with $c_0 = c$, $c_k = c'$, such that two consecutive collapses differ only by one hyperplane, and there exists a sequence $\iota_i\colon S_i\to \SS$ of combinatorial isomorphisms such that $[(\iota_i)_*\circ (c_i)_*\circ (\iota_*\circ c_*)^{-1}]\in \unt(A_\Gamma;\,\mathcal{G})$ for all $i$. Letting $m_i$ be the marking on $S_i$ making $c_i$ marking-preserving, $[(\iota_i)_*\circ m_i^{-1}] = [(\iota_i)_*\circ (c_i)_*\circ (\iota_*\circ c_*)^{-1}][\iota_*\circ m^{-1}]\in \unt(A_\Gamma;\,\mathcal{G})$, proving that $[S_i,m_i]\in S^\mathcal{G}$. Up to removing loops if $[S_i,m_i]$ was equal to $[S_j,m_j]$, this yields a Whitehead path in $K^\mathcal{G}$ from $\sigma$ to $\tau$. Moreover, each marked Salvetti of the Whitehead path is joined by an edge to $\alpha$, and each edge of the path factors such an edge to $\alpha$.

    Since only finitely many edges are incident to $\alpha$ in $K^\mathcal{G}$, there are only finitely many Whitehead paths in $K^\mathcal{G}$ from $\sigma$ to $\tau$ satisfying the previous assumption. Pick one such path where the maximum norm attained by a marked Salvetti of the path is as small as possible. We will still use the notations $[S_i,m_i]$ for vertices of this path, with $c_i\colon X\to S_i$ preserving markings and $\iota_i\colon S_i\to \SS$ combinatorial. The maximal marked Salvetti is unique by injectivity of $N$, and it is different from $[S_k,m_k] = \tau$ since $N(\sigma)>N(\tau)$. Assume it is $[S_i,m_i]$ with $0<i<k$. In particular, $N([S_{i-1},m_{i-1}])<N([S_i,m_i])$ and $N([S_{i+1},m_{i+1}])<N([S_i,m_i])$ (remember that the norm is injective on marked Salvettis).

    The two edges of the Whitehead path at $[S_i,m_i]$ are factoring the edge joining $[S_i,m_i]$ to $\alpha$, hence are compatible. By Lemma~\ref{peak}, the portion of the path between $[S_{i-1},m_{i-1}]$ and $[S_{i+1},m_{i+1}]$ can be replaced either by a single Whitehead move, making the path shorter, or by a concatenation of two Whitehead moves with middle marked Salvetti still in $S^\mathcal{G}$ and lesser than $[S_i,m_i]$ making the maximum of the path smaller (still by injectivity of $N$). This is a contradiction in both cases. Therefore, the maximal marked Salvetti is $[S_0,m_0]$ and the first edge of the path is $N\mathcal{G}$-reductive at $\sigma$ and factors $e$.
\end{proof}

\begin{lm}[see \cite{untwistedos}, Lemma~6.17, Case 1]
\label{higginslyndon}
    Let $e_1,e_2$ be $N\mathcal{G}$-reductive elementary edges incident to the same marked Salvetti $\sigma\in S^\mathcal{G}$. There exists an $N\mathcal{G}$-reductive elementary edge $e_3$ incident to $\sigma$ that is compatible both with $e_1$ and $e_2$.
\end{lm}

To prove this lemma, we will use the proof of Lemma~6.17 in \cite{untwistedos} and the vocabulary of Whitehead partitions. We need first a technical lemma to relate some Whitehead partitions appearing in Lemma~6.16 of \cite{untwistedos}, with the group $\unt(A_\Gamma;\,\mathcal{G})$. Recall that a based Whitehead partition is entirely determined by its basepoint and one side (the link is the link of the basepoint and the other side is the complement). Recall also that two based Whitehead partitions $(\mathbf{P},v)$ and $(\mathbf{Q},w)$ that are distinct (even up to change of sides and inversion of basepoint) are non-compatible if and only if none of $P\cap Q, P\cap Q^*, P^*\cap Q, P^*\cap Q^*$ are empty and $v,w$ are equal or do not commute.

\begin{lm}[see \cite{untwistedos}, Lemma~6.16]
\label{relativepartition}
Let $(\mathbf{P},v)$, $(\mathbf{Q},w)$ be two distinct non-compatible based Whitehead partitions whose corresponding Whitehead automorphisms are in $\unt(A_\Gamma;\,\mathcal{G})$.
\begin{enumerate}
\item If $v\in double(Q)$ and $w^{-1}\in P$, the following define two based Whitehead partitions compatible with $\mathbf{P}$ and $\mathbf{Q}$ whose corresponding automorphisms are in $\unt(A_\Gamma;\,\mathcal{G})$:
\[(\mathbf{R}, w^{-1}),\,\text{with } R = P\cap Q^*\qquad (\mathbf{S}, v^{-1}),\,\text{with } S = P^*\cap Q\]
\item If $v\in single(\mathbf{Q})$, $w\in single(\mathbf{P})$, and $v\in Q$, the following define two based Whitehead partitions compatible with $\mathbf{P}$ and $\mathbf{Q}$ whose corresponding automorphisms are in $\unt(A_\Gamma;\,\mathcal{G})$:
\[(\mathbf{T}, v),\,\text{with } T = P\cap Q\qquad (\mathbf{U}, v^{-1}),\,\text{with } U = P^*\cap Q^*\]
\end{enumerate}
\end{lm}
Note that in both cases, symmetrical results hold if the sides of $\mathbf{P}$ or $\mathbf{Q}$ are exchanged (with basepoint replaced by their inverses). The reader will check easily that this covers all the cases of Whitehead partitions obtained in \cite{untwistedos}, Lemma~6.16.

\begin{proof}
The fact that these expressions define indeed based Whitehead partitions compatible with $\mathbf{P}$ and $\mathbf{Q}$ is the exact content of \cite{untwistedos}, Lemma~6.16. The only thing left to prove is that the corresponding automorphisms are in $\unt(A_\Gamma;\,\mathcal{G})$. To do so, we check Condition $(2.)$ from Lemma~\ref{whitehead} for some arbitrary $A_{\Delta_i}\in\mathcal{G}$. Note that, since each partition $\mathbf{R},\mathbf{S},\mathbf{T},\mathbf{U}$ has one side contained both in a side of $\mathbf{P}$ and a side of $\mathbf{Q}$, that side does not intersect the symmetric sets $lk(\mathbf{P})$ and $lk(\mathbf{Q})$. In particular, $single(\mathbf{R})$, $single(\mathbf{S})$, $single(\mathbf{T})$ and $single(\mathbf{U})$ are all disjoint from $lk(\mathbf{P})$ and $lk(\mathbf{Q})$.

Several cases arise for Assertion~$(1)$.
\begin{itemize}
\item Assume $v,w\in \Delta_i$. The result is obvious for both $\mathbf{R}$ and $\mathbf{S}$.

\item Assume $v\in \Delta_i$ and $w\notin \Delta_i$. The result is obvious for $\mathbf{S}$. Since $v\in double(Q)\cap \Delta_i$, $double(Q^*)\cap \Delta_i$ is empty by Lemma~\ref{whitehead}. The inclusions $single(\mathbf{R})\subseteq single(\mathbf{Q})\cup double(Q^*)$ and $double(R)\subseteq double(Q^*)$ provide the result for $\mathbf{R}$.

\item Assume $v \notin \Delta_i$ and $w\in \Delta_i$. The result is obvious for $\mathbf{R}$. By Lemma~\ref{whitehead}, $single(\mathbf{P})\cap \Delta_i$ is empty, in particular $w\notin single(\mathbf{P})$. Hence $w\in double(P)\cap \Delta_i$, and $double(P^*)\cap \Delta_i$ is empty. As above, the inclusions $single(\mathbf{S})\subseteq single(\mathbf{P})\cup double(P^*)$ and $double(S)\subseteq double(P^*)$ provide the result for $\mathbf{S}$.

\item Assume finally $v,w\notin \Delta_i$. The inclusion $single(\mathbf{R})\subseteq single(\mathbf{P})\cup single(\mathbf{Q})$ proves that $single(\mathbf{R})\cap \Delta_i$ is empty. Now, for the sake of contradiction, let $x_1\in double(R)\cap \Delta_i$ and $x_2\in double(R^*)\cap \Delta_i$. Note that since $x_2\in double(R^*)$, $x_2\notin lk(\mathbf{R}) = lk(\mathbf{Q})$. Several subcases arise.
\begin{enumerate}
    \item[a.] If $x_2\in single(\mathbf{P})\cup single(\mathbf{Q})$, $x_2 \in \Delta_i$ is absurd. 
    \item[b.] If $x_2\in double(P^*)\cup double(Q)$, since $x_1\in double(R)\subseteq double(P)\cap double(Q^*)$, $x_1,x_2\in \Delta_i$ is absurd.
    \item[c.] Otherwise, $x_2\notin single(\mathbf{P})\cup double(P^*)$ and $x_2\in double(Q^*)$. Thus, since $x_2\notin double(R)$, $x_2\notin double(P)$. Hence $x_2\in lk(\mathbf{P})$ and $x_2$ commutes with $v$. However, $x_2\in double(Q^*)$ and $v\in double(Q)$, hence they cannot commute, a contradiction.
\end{enumerate}
This proves the desired result for $\mathbf{R}$. The proof for $\mathbf{S}$ is very similar, exchanging $(\mathbf{P},v)$ and $(\mathbf{Q},w)$:
The inclusion $single(\mathbf{S})\subseteq single(\mathbf{P})\cup single(\mathbf{Q})$ proves that $single(\mathbf{S})\cap \Delta_i$ is empty. Now, for the sake of contradiction, let $y_1\in double(S)\cap \Delta_i$ and $y_2\in double(S^*)\cap \Delta_i$. Note that since $y_2\in double(S^*)$, $y_2\notin lk(\mathbf{S}) = lk(\mathbf{P})$. Several subcases arise.
\begin{enumerate}
    \item[a.] If $y_2\in single(\mathbf{Q})\cup single(\mathbf{P})$, $y_2 \in \Delta_i$ is absurd.
    \item[b.] If $y_2\in double(Q^*)\cup double(P)$, since $y_1\in double(S)\subseteq double(Q)\cap double(P^*)$, $y_1,y_2\in \Delta_i$ is absurd.
    \item[c.] Otherwise, $y_2\notin single(\mathbf{Q})\cup double(Q^*)$ and $y_2\in double(P^*)$ in this case. Thus, since $y_2\notin double(S)$, $y_2\notin double(Q)$. Hence $y_2\in lk(\mathbf{Q})$ and $y_2$ commutes with $w$. This time, since $y_2\in double(P^*)$ and $w^{-1}\in P$, they cannot commute, a contradiction again.
\end{enumerate}

\end{itemize}

Several cases arise for Assertion $(2)$ as well. We prove only the result for $\mathbf{T}$, the case of $\mathbf{U}$ being entirely symmetrical.
\begin{itemize}
\item If $v\in\Delta_i$, the result is obvious.
\item If $w\in \Delta_i$, since $w\in single(\mathbf{P})$, $v\in\Delta_i$ and the result is again obvious.
\item Finally, if $v,w\notin \Delta_i$, the argument is identical to the last case of the proof of Assertion~$(1)$ for $\mathbf{S}$, switching the sides of $\mathbf{P}$ (the proof of this case does not use the fact that $v\in double(Q)$.)
\end{itemize}
\end{proof}

\begin{proof}[Proof of Lemma~\ref{higginslyndon}]
    Let $\varphi\in\unt(A_\Gamma;\,\mathcal{G})$ such that $\sigma = [\SS,\varphi]$. Let $e'_1$ (resp. $e'_2$) be an edge completing $e_1$ (resp. $e_2$) into an $N\mathcal{G}$-reductive Whitehead move from $\sigma$. Using the fact that any collapse map comes from a collapse of partitions in a blow-up of $\SS$ (\cite{spatial}, Proposition~5.7), let $(\mathbf{P},v)$ (resp. $(\mathbf{Q},w)$) be a $\varphi$-$N\mathcal{G}$-reductive based Whitehead partition representing the Whitehead move $e_1e'_1$ (resp. $e_2e'_2$) at $\sigma$. Let $c_\mathbf{P}\colon (\SS^\mathbf{P},M_\mathbf{P})\to (\SS, \varphi)$, $c_v\colon (\SS^\mathbf{P},M_\mathbf{P})\to (\SS^\mathbf{P}_v, m)$ (resp. $c_\mathbf{Q}\colon (\SS^\mathbf{Q},M_\mathbf{Q})\to (\SS, \varphi)$, $c_w\colon (\SS^\mathbf{Q},M_\mathbf{Q})\to (\SS^\mathbf{Q}_w, m')$) be marking-preserving collapses of the hyperplanes with the corresponding labels. If $\mathbf{P}$ and $\mathbf{Q}$ are compatible, the edges $e_1$ and $e_2$ both factor some edge representing the collapse $\SS^{\{\mathbf{P},\mathbf{Q}\}}\to\SS$ with appropriate markings. Otherwise, $\mathbf{P}$ and $\mathbf{Q}$ are not compatible. By preservation of markings, $[m\circ (c_v)_*\circ (c_\mathbf{P})_*^{-1}]=[m'\circ (c_w)_*\circ (c_\mathbf{Q})_*^{-1}]=\varphi$. By assumption and by $\unt(A_\Gamma;\,\mathcal{G})$-invariance of $S^\mathcal{G}$, $\varphi^{-1}\cdot [\SS^\mathbf{P}_v,m] = [\SS^\mathbf{P}_v, (c_\mathbf{P})_*\circ (c_v)_*^{-1}] $ and $\varphi^{-1}\cdot [\SS^\mathbf{Q}_w,m'] = [\SS^\mathbf{Q}_w, (c_\mathbf{Q})_*\circ (c_w)_*^{-1}] $ are in $S^\mathcal{G}$. By Lemma~\ref{whitehead}, the outer Whitehead automorphisms corresponding to the based Whitehead partitions $(\mathbf{P},v)$ and $(\mathbf{Q},w)$ are in $\unt(A_\Gamma;\,\mathcal{G})$.
    
    Moreover, by Remark~\ref{collapseeverywhere} and Lemma~\ref{changenorm}, $\mathbf{0}<_{lex} N([\SS,\varphi]) - N([\SS^\mathbf{P}_v,m]) = N_v(\SS,\varphi) - N_\mathbf{P}(\SS^\mathbf{P},M_\mathbf{P})$, hence $N_v(\SS,\varphi) > N_\mathbf{P}(\SS^\mathbf{P},M_\mathbf{P})$ and $\mathbf{0}<_{lex} N([\SS,\varphi]) - N([\SS^\mathbf{Q}_w,m']) = N_w(\SS,\varphi) - N_\mathbf{Q}(\SS^\mathbf{Q},M_\mathbf{Q})$, hence $N_w(\SS,\varphi) > N_\mathbf{Q}(\SS^\mathbf{Q},M_\mathbf{Q})$.

    Now we reuse the proof of Lemma~6.17 in \cite{untwistedos} (Case 1). The order of entries in our norm is different, but we obtained the inequalities $|\mathbf{P}|_\sigma-|v|_\sigma<\mathbf{0}$ and $|\mathbf{Q}|_\sigma-|w|_\sigma<\mathbf{0}$ with our ordering of entries by the argument above, and the reader can check that all the other involved inequalities in the proof of \cite{untwistedos} hold term by term of the sequence. Thus there exist a based Whitehead partition $(\mathbf{V},b)$ obtained from Lemma~6.16 in \cite{untwistedos} such that $\mathbf{V}$ is both compatible with $\mathbf{P}$ and $\mathbf{Q}$, and $|\mathbf{V}|_\sigma - |b|_\sigma <\mathbf{0}$ with our ordering of entries, i.e.~$N_b(\SS,id)>N_{\mathbf{V}}(\SS^\mathbf{R},M_\mathbf{V})$, where $M_\mathbf{V}$ is the marking on $\SS^\mathbf{V}$ making the collapse of $\mathbf{V}$ marking-preserving. Let $m''$ be the marking on the range of the collapse $c_b\colon \SS^\mathbf{V}\to \SS^\mathbf{V}_b$ of the hyperplane labeled $b$ making $c_b$ marking-preserving, i.e.~$m''\circ (c_b)_*\circ (c_\mathbf{V})_*^{-1} = M_{\mathbf{V}}\circ (c_\mathbf{V})_*^{-1} = \varphi$. By Remark~\ref{collapseeverywhere} and Lemma~\ref{changenorm}, $N([\SS^\mathbf{V}_b,m'']) = N([\SS, \varphi]) + N_{\mathbf{V}}(\SS^\mathbf{V},M_\mathbf{V}) - N_b(\SS,\varphi)< N([\SS,\varphi])$. The edge $e_3$ of $K_\Gamma$ representing the collapse $c_\mathbf{V}\colon \SS^\mathbf{V}\to \SS$ is incident to $\sigma$ and part of a Whitehead move whose other endpoint $[\SS^\mathbf{V}_b,m'']$ has lesser norm than $\sigma$. Furthermore, the fact that $\mathbf{V}$ is compatible with both $\mathbf{P}$ and $\mathbf{Q}$ proves that $e_3$ is compatible with both $e_1$ and $e_2$.

    Finally, $(\mathbf{V},b)$ was obtained from \cite{untwistedos} Lemma~6.16 from $(\mathbf{P},v)$ and $(\mathbf{Q},w)$. By Lemma~\ref{relativepartition}, $(\mathbf{V},b)$ induces an outer Whitehead automorphism in $\unt(A_\Gamma;\,\mathcal{G})$. Therefore $[\SS^\mathbf{V}_b,m''] = \varphi\cdot [\SS^\mathbf{V}_b,(c_\mathbf{V})_*\circ (c_b)_*^{-1}]\in S^\mathcal{G}$ by Lemma~\ref{whitehead}, finishing the proof that $e_3$ is $N\mathcal{G}$-reductive.
\end{proof}

\begin{lm}[see \cite{untwistedos}, Theorem~6.18]
\label{peakred}
    Let $\sigma,\tau_1,\tau_2$ be three distinct marked Salvettis in $S^\mathcal{G}$. Assume there exist $N\mathcal{G}$-reductive Whitehead moves from $\sigma$ to $\tau_1$ and from $\sigma$ to $\tau_2$. Then there exists a Whitehead path in $K^\mathcal{G}$ from $\tau_1$ to $\tau_2$ passing only through marked Salvettis $\tau$ with $N(\tau)<N(\sigma)$.
\end{lm}
\begin{proof}
    Let $e_1$, $e_2$ be the $N\mathcal{G}$-reductive edges at $\sigma$ starting the Whitehead moves to $\tau_1$ and $\tau_2$ respectively. If $e_1$ and $e_2$ are compatible, Lemma~\ref{peak} concludes. If $e_1$ and $e_2$ are not compatible, by Lemma~\ref{higginslyndon}, there exists an $N\mathcal{G}$-reductive edge $e_3$ at $\sigma$ compatible with both $e_1$ and $e_2$. Let $\tau_3$ be a marked Salvetti in $S^\mathcal{G}$ adjacent to an endpoint of $e_3$, with $N(\tau_3)<N(\sigma)$. Applying the compatible case to both $\tau_1, \tau_3$ and $\tau_3,\tau_2$, then concatenating the obtained Whitehead paths yields the result.
\end{proof}

\begin{cor}[see \cite{untwistedos}, Corollary~6.20]
\label{reduction}
    Every marked Salvetti in $S^\mathcal{G}\setminus\{\sigma_0\}$ admits an incident $N\mathcal{G}$-reductive elementary edge.
\end{cor}
\begin{proof}
    Let $\sigma\in S^\mathcal{G}\setminus \{\sigma_0\}$. By Lemma~\ref{connected}, there exists a Whitehead path in $K^\mathcal{G}$ joining $\sigma$ to $\sigma_0$. Since the Whitehead path does not go through the same marked Salvetti twice and the norm is injective, the Whitehead path reaches a marked Salvetti of maximum norm only once. If the marked Salvetti of maximum norm is not an endpoint of the path, use Lemma~\ref{peakred} with this maximal marked Salvetti and its two neighbors to alter the path, reducing the maximum norm attained while staying in $K^\mathcal{G}$. By Lemma~\ref{wellorder}, this alteration can only be performed finitely many times, after which either $\sigma$ or $\sigma_0$ is maximal. Since $N(\sigma_0)<N(\sigma)$ by minimality, $\sigma$ is maximal and the first edge of the path is $N\mathcal{G}$-reductive at $\sigma$. 
\end{proof}

For this last preliminary lemma, we need the vocabulary of Whitehead partitions again.

\begin{lm}[see \cite{untwistedos}, Proposition~6.23]
\label{pushing}
    Let $\sigma = [\SS,\varphi]\in S^\mathcal{G}\setminus \{\sigma_0\}$ with $\varphi\in \unt(A_\Gamma;\,\mathcal{G})$. Let $(\mathbf{M}, m)$ be a based Whitehead partition satisfying the following assumptions:
    \begin{itemize}
        \item $(\mathbf{M},m)$ is $\varphi$-$N\mathcal{G}$-reductive (such a based Whitehead partition always exists by Corollary~\ref{reduction}).
        \item $lk(\mathbf{M})$ is maximal among all based Whitehead partitions satisfying the previous property
        \item $N(\tau)$, is minimal among all based Whitehead partitions satisfying the two previous properties, where $\tau$ is the other endpoint of the Whitehead move corresponding to $(\mathbf{M},m)$, starting at $(\SS,\varphi)$.
    \end{itemize}
    Let $(\mathbf{P},p)$ be a $\varphi$-$N\mathcal{G}$-reductive Whitehead partition. Assume $\mathbf{P}$ and $\mathbf{M}$ are not compatible, so that $m\notin lk(\mathbf{P})$. Assume up to possibly exchanging the sides of $\mathbf{P}$ that $m\in P$. Then one of the quadrants $P^*\cap M$ or $P^*\cap M^*$ determines a $\varphi$-$N\mathcal{G}$-reductive based Whitehead partition $(\mathbf{P}_0,p_0)$ (compatible with both $\mathbf{M}$ and $\mathbf{P}$), with $lk(\mathbf{P}_0) = lk(\mathbf{P})$.
\end{lm}

\begin{proof}
    The proof is drawn from \cite{untwistedos}, Proposition~6.23. Once again, we need to reorder the entries of the norms, but all the lexicographic inequalities are either true term-by-term or a consequence of the assumptions of the lemma. Since $(\mathbf{P}_0,p_0)$ is obtained by applying \cite{untwistedos}, Lemma~6.16 with $(\mathbf{M},m)$ and $(\mathbf{P},p)$ which both induce Whitehead automorphisms in $\unt(A_\Gamma;\,\mathcal{G})$, $(\mathbf{P}_0, p_0)$ also induces a Whitehead automorphism in $\unt(A_\Gamma;\,\mathcal{G})$ by Lemma~\ref{relativepartition}. This finishes the proof that $(\mathbf{P}_0,p_0)$ is $\varphi$-$N\mathcal{G}$-reductive.
\end{proof}

\begin{pro}[see \cite{untwistedos}, proof of Theorem~6.24]
\label{contractintersect}
Let $\sigma\in S^\mathcal{G}\setminus \{\sigma_0\}$. The intersection $K^\mathcal{G}_{<\sigma}\cap st(\sigma)$ is contractible.
\end{pro}

\begin{proof}
Vertices in $K^\mathcal{G}_{<\sigma}\cap st(\sigma)$ are exactly the $\alpha$ that are adjacent to $\sigma$ and to some marked Salvetti $\tau\in S^\mathcal{G}$ with $N(\tau)<N(\sigma)$. Such a vertex exists by Corollary~\ref{reduction}. The cell structure on the $K^\mathcal{G}_{<\sigma}\cap st(\sigma)$ corresponds to the induced poset of collapses.

The rest of the proof is identical to \cite{untwistedos}: by Lemma~\ref{factoring} and Quillen's poset lemma, $K^\mathcal{G}_{<\sigma}\cap st(\sigma)$ deformation retracts to its subcomplex $K'$ spanned by vertices $\alpha'$ such that every elementary edge factoring the edge from $\alpha'$ to $\sigma$ is $N\mathcal{G}$ reductive. Choose then $(\mathbf{M},m)$ a based Whitehead partition satisfying the assumptions of Lemma~\ref{pushing}, and build a deformation retraction of $K'$ to a point exactly as in the proof of Theorem~6.24 in \cite{untwistedos}.
\end{proof}

\begin{cor}
For every $\sigma\in S^\mathcal{G}$ different from $\sigma_0$, the complex $K^\mathcal{G}_{<\sigma}$ is contractible. The complex $K^\mathcal{G}$ is also contractible.
\end{cor}
This concludes the proof of Theorem~\ref{mccoolcyclic}, since $K^\mathcal{G}_{\mathcal{H}^t}$ is equal either to $K^\mathcal{G}$ or to $K^\mathcal{G}_{<\sigma_1}$.
\begin{proof}
The proof is by transfinite induction. Let $P(\sigma)$ be the property "Every homotopy group of $K^\mathcal{G}_{<\sigma}$ is trivial". The property holds for the successor $\tau$ of $\sigma_0$ in the well-ordering since $K_{<\tau} = st(\sigma_0)$ is contractible. To prove preservation under successors assume $K^\mathcal{G}_{<\sigma}\cap st(\sigma)$ and $st(\sigma)$ are contractible subcomplexes of $K^\mathcal{G}$. Then $K^\mathcal{G}_{<\sigma}\cup st(\sigma)$ is homotopy equivalent to the topological quotient $(K^\mathcal{G}_{<\sigma}\cup st(\sigma))/st(\sigma) = K^\mathcal{G}_{<\sigma}/(K^\mathcal{G}_{<\sigma}\cap st(\sigma))$ which is itself homotopy equivalent to $K^\mathcal{G}_{<\sigma}$ by Proposition~\ref{contractintersect}. Finally, the property is preserved under increasing limits because homotopy classes are compactly supported and all the objects are CW-complexes. Hence $P(\sigma)$ is true for every $\sigma\in S^\mathcal{G}$ different from $\sigma_0$, and $K^\mathcal{G}_{<\sigma}$ is contractible by Whitehead's theorem on CW-complexes. Finally, if $S^\mathcal{G}$ has a maximum $\sigma_{max}$, $K^\mathcal{G} = K^\mathcal{G}_{<\sigma_{max}}\cup st(\sigma_{max})$, and otherwise $K^\mathcal{G} = \displaystyle\bigcup_{\sigma\in S^\mathcal{G}} K^\mathcal{G}_{<\sigma}$. Thus $K^\mathcal{G}$ is contractible by the same arguments as above.
\end{proof}

\section{Adding arbitrary fixed subgroups}
\label{fixedfg}

This concluding section generalizes Theorem~\ref{mccoolcyclic}, starting with the following consequence of Proposition~\ref{bounded_displacement}:

\begin{lm}
\label{finite_index_general}
Let $H<A_\Gamma$ be finitely generated. There exists $\mathcal{K} = (K_1,\dots, K_L)$ a family of cyclic subgroups of $A_\Gamma$ such that $\out(A_\Gamma;\,\{H\}^t)$ is a subgroup of finite index of $\out(A_\Gamma;\, \mathcal{K}^t)$.
\end{lm}

\begin{proof}
Fix a generating system $H = \gen{a_1,\dots, a_{k}}$ and set for $(K_i)_{1\leq i\leq L}$ the family of all cyclic subgroups generated by group elements of the form $a_i$, or $a_{i}a_{j}$ with $i<j$. As every $K_i$ is a subgroup of $H$, an outer automorphism that is inner on $H$ is necessarily inner on every $K_i$, proving the inclusion $\out(A_\Gamma;\,\{H\}^t)\subseteq \out(A_\Gamma;\,\mathcal{K}^t)$.

Besides, given $[\varphi] \in \mathrm{Out}(A_\Gamma;\,\mathcal{K}^t)$, the automorphism $\varphi$ maps each $a_i$ and each $a_ia_j$ with $i<j$ to conjugates of themselves. In particular $\ell(\varphi(a_i))=\ell(a_i)$ and $\ell(\varphi(a_ia_j)) = \ell(a_ia_j)$.

Now apply Proposition~\ref{bounded_displacement} to $X=\USS$ and the family $(\varphi(a_1),\dots, \varphi(a_k))$. The bound $M$ depends only on the $a_i$ and not on $\varphi$, and there exists $x_\varphi\in X^{(0)}$ such that $d(x_\varphi, \varphi(a_i)x_\varphi)\leq M$ for every $i$. As $X$ has only one orbit of vertices, there is a group element $g_\varphi$ such that $x_{id} = g_\varphi x_\varphi$. Thus, conjugating $\varphi$ by $g_\varphi^{-1}$, pick a representative $\varphi_c$ for each class $c\in \mathrm{Out}(A_\Gamma;\,\mathcal{K}^t)$ satisfying $d(x_{id}, \varphi_c(a_i)x_{id}) \leq M$ for every $i$.

Finally, let $B$ be the finite vertex set of the closed $\ell^1$-ball in $\USS$ of radius $M$ centered at $x_{id}$, and consider the map:
\[\begin{aligned}\mathrm{Out}(A_\Gamma;\,\mathcal{K}^t)&\to B^k\\
c&\mapsto (\varphi_c(a_i)x_{id})_{1\leq i\leq k}
\end{aligned}\]

If two classes $c,c'$ have the same image under this map, then $\varphi_c^{-1}\circ \varphi_{c'}$ fixes every $a_i$ because the action is free. Since, $H$ is generated by the $a_i$, this means that $c^{-1}c'\in \mathrm{Out}(A_\Gamma;\,\{H\}^t)$. Hence $\out(A_\Gamma;\,\{H\}^t)$ has at most $\card{B^k}$ left cosets in $\out(A_\Gamma;\,\mathcal{K}^t)$.
\end{proof}

\begin{thm}
\label{mccoolgeneral}
Let $\mathcal{G} = (A_{\Delta_1},\dots, A_{\Delta_m})$ be a collection of standard subgroups of $A_\Gamma$ and let $\mathcal{H} = (H_1,\dots,H_n)$ be a finite collection of arbitrary finitely generated subgroups of $A_\Gamma$. There exists a contractible subcomplex $K^{\mathcal{G}}_{\mathcal{H}^t}$ of the spine of untwisted outer space $K_\Gamma$ on which $\unt(A_\Gamma; \mathcal{G},\mathcal{H}^t)$ acts properly and cocompactly. In particular, this group is of type VF.
\end{thm}

\begin{proof}
For each $H_j$, set $\mathcal{K}_j$ as in Lemma~\ref{finite_index_general}, and let $\mathcal{K}$ be the union of the $\mathcal{K}_j$. The group $\unt(A_\Gamma;\, \mathcal{G},\mathcal{H}^t) = \unt(A_\Gamma;\,\mathcal{G})\cap \displaystyle \bigcap_j \out(A_\Gamma;\, \{H_j\}^t)$ is a finite-index subgroup of $\unt(A_\Gamma;\,\mathcal{G})\cap \displaystyle \bigcap_j \out(A_\Gamma;\, \mathcal{K}_j^t) = \unt(A_\Gamma;\,\mathcal{G},\mathcal{K}^t)$. Hence, it acts properly and cocompactly on the contractible subcomplex $K^\mathcal{G}_{\mathcal{K}^t}$ of $K_\Gamma$ from Theorem~\ref{mccoolcyclic}.
\end{proof}

We conjecture the following, that would strengthen Theorem~\ref{mccoolgeneral} by allowing finitely many arbitrary abelian subgroups in $\mathcal{G}$ in addition to the standard subgroups. Note that abelian subgroups of $A_\Gamma$ have rank bounded by the cohomological dimension of $A_\Gamma$, hence are finitely generated.
\begin{conj}
Let $H<A_\Gamma$ be abelian, there exists $\mathcal{K} = (K_1,\dots,K_L)$ a family of cyclic subgroups of $A_\Gamma$ such that $\unt(A_\Gamma;\, \{H\})$ is a subgroup of finite index of $\unt(A_\Gamma;\, \mathcal{K}^t)$.
\end{conj}

We also make the following conjecture.
\begin{conj}
The set $\{Fix(\varphi)\mid \varphi\in \uaut(A_\Gamma)\}$, ordered by inclusion, contains no infinite chain.
\end{conj}

Following \cite{guirardel_mccool_2015}, this would prove that, in Theorem~\ref{mccoolgeneral}, assuming the elements of $\mathcal{H}$ finitely generated is unnecessary.

\bibliographystyle{plain}
\bibliography{VF}
\bigskip

\begin{flushleft}
Adrien Abgrall
\\
Univ Rennes, CNRS, IRMAR - UMR 6625, F-35000 Rennes, France.
\\
E-mail: \href{mailto:adrien.abgrall@univ-rennes.fr}{\texttt{adrien.abgrall@univ-rennes.fr}}
\end{flushleft}

\end{document}